\documentclass[smallextended]{svjour3}       % onecolumn (second format)
\smartqed  % flush right qed marks, e.g. at end of proof
\usepackage{graphicx}
\usepackage{amsmath,amssymb,amsfonts,latexsym,subfigure,hyperref,enumitem,comment}
\makeatletter
\def\wbar{\accentset{{\cc@style\underline{\mskip8mu}}}}

\makeatother

\newcommand{\f}{f}
\newcommand{\g}{g}
\newcommand{\M}{\ensuremath{\mathcal{M}}}
\newcommand{\PA}{\mathcal{P}}

\newcommand{\R}{\ensuremath{\mathbb{R}}}
\newcommand{\A}{\ensuremath{\mathcal{A}}}

\def\D{{\mathcal D}}

\renewcommand{\vec}[1]{{#1}}
\def\CH{C\!H^+}
\def\CVH{CV\!H^+}
\DeclareMathOperator*{\Ker}{\mathrm{Ker}}
\DeclareMathOperator*{\Ran}{\mathrm{Ran}}

\usepackage{xypic} 
\usepackage{xcolor}

\def\makeheadbox{\relax}

\begin{document}

\title{Nonlinear Spectral Duality%: Theory and application to graph cuts and convex geometry%\thanks{Grants or other notes
}
% \subtitle{Do you have a subtitle?\\ If so, write it here}

\titlerunning{Nonlinear Spectral Duality}        % if too long for running head

\author{Francesco Tudisco       \and
        Dong Zhang %etc.
}


\institute{F. Tudisco \at
              School of Mathematics\\
    Gran Sasso Science Institute\\
    67100 L'Aquila (Italy) \\
    \email{francesco.tudisco@gssi.it}       %  \\
           \and
           D. Zhang \at
           LMAM and School of Mathematical Sciences\\
        Peking University\\ 
      100871 Beijing  (China)\\
        \email{dongzhang@math.pku.edu.cn}
}

\date{Received: date / Accepted: date}

\maketitle

\begin{abstract}
Nonlinear eigenvalue problems for pairs of homogeneous convex functions are particular nonlinear constrained optimization problems that arise in a variety of settings, including graph mining, machine learning, and network science. 
By considering different notions of duality transforms from both classical and recent convex geometry theory, in this work we show that one can move from the primal to the dual nonlinear eigenvalue formulation maintaining the spectrum, the variational spectrum as well as the corresponding multiplicities unchanged. These  nonlinear spectral duality properties can be used to transform the original optimization problem into various alternative and possibly more treatable dual problems. We illustrate the use of nonlinear spectral duality in a variety of example settings involving optimization problems on graphs, nonlinear Laplacians, and distances between convex bodies.   
\end{abstract}

\section{Introduction and motivation}\label{sec:introduction}

 The critical values and critical points of the ratio of convex homogeneous functions $f(x)/g(x)$ define (sometimes only a part of) the nonlinear spectrum of the functions pair $(f,g)$. This type of nonlinear eigenvalue problem appears in a wide range of applications. Examples include graph-based machine learning, where the spectral properties of different notions of nonlinear graph and hypergraph Laplacian operators play a central role in unsupervised and semi-supervised classification algorithms \cite{bresson2014multi,Bhuler,calder2018game,Elmoataz,flores2022analysis,prokopchik2022,slepcev2019analysis,Tudisco1}; the approximation of matrix and tensor norms \cite{GHT20,gautier2019unifying,nguyen2017efficient}; the solution of the Gross-Pitaevskii equation in quantum chemistry \cite{cai2018eigenvector,saad2010numerical,upadhyaya2021density}; the identification and analysis of relevant mesoscopic structures in complex networks, such as central nodes, communities and core-periphery \cite{boyd2018simplified,hu2013method,tudisco2018core,tudisco2021nodeandedge,tudisco2018community}; the optimization of polynomials and generalized polynomials on the unit sphere \cite{gautier2016globally,gautier2019unifying,zhou2012nonnegative}. 
 
 A number of complications arise when moving from the  classical matrix eigenvalue problem to the  nonlinear one, starting from the fact that  the number of eigenvalues and eigenvectors is no longer bounded by the space dimension. However, in most cases one can use the Lusternik-Schnirelmann theory combined with the Krasnoselski genus to define a sequence of variational eigenvalues by means of a Courant-Fisher-like minmax characterization. This subset of variational eigenvalues has very useful properties in most application settings. However, unlike the linear case,  evaluating, computing,  or approximating the variational  eigenvalues is in general a very challenging problem in the nonlinear case, which boils down to a nonsmooth optimization problem for pairs of homogeneous convex functions. % $(f,g)$. 
 
 %%%
 % FT: Say somewhere we develop the nonlinear spectral duality theory
 %%%
 
 In this paper, we focus on the family of function pairs $(f,g)$ that, on top of being  homogeneous and  convex, are nonnegative and thus have a linear kernel. These properties are very common in a range of applications, as we will further detail in Section \ref{sec:app}.

 % For example,  nonlinear eigenvalue problems for homogeneous functions are investigated in \cite{ABCM01,BCN02,hein}, related nonlinear spectral decomposition are  proposed in \cite{BBCN21,BHPG21,BGMEC16}, and different algorithms for the computation of the eigenpairs are considered in \cite{HyndLindgren17,FAGP19,BungertBurger22}. 

 For this type of functions, we define three duality transforms obtained by adapting the norm duality, the Fenchel's convex conjugate  (i.e. Legendre transform) and the  polarity transform (or  $\A$-transform) \cite{AM11,AR17}.
Thus, we 
 provide three main results showing that the variational spectrum as well as its multiplicities are invariant under these duality transforms. These novel theoretical properties have a number of useful  implications as they allow us to move from a given nonlinear eigenvalue problem to several new dual problems which, depending on the particular setting, may result in a more treatable optimization problem or may reveal useful properties that are difficult to observe and to prove using the primal eigenvalue formulation. 

 For example, if $f=g=\|\cdot\|$ are norms, convergence guarantees for the fixed point iteration method to compute $\max f(x)/g(x)$ may be obtained using the dual pair, while the same method may fail to converge for the primal problem~\cite{GHT20}.
 Similarly, a variety of  established algorithms for nonlinear eigenproblems such as the inverse iterations \cite{HyndLindgren17,jarlebring2014inverse}, the family of RatioDCA methods \cite{NIPS2011_193002e6,tudisco2018community}, the MBO energy landscape and active set search methods for graph total variation \cite{boyd2018simplified,cristofari2020total,hu2013method}, or  the continuous gradient-flow approach \cite{BungertBurger22,FAGP19},  can be directly transferred to the dual eigenvalue equations. The resulting dual iteration or dual flow can be used to solve the optimization of the primal eigenvalue problem and may behave better in practice. 
 Several more specific application settings where nonlinear spectral duality may be used are illustrated in Section \ref{sec:app}. Some of the example settings there discussed contain new results we obtain as a consequence of our spectral duality theory.  %are then accompanied by a number of application settings where nonlinear spectral duality may be used. 

Our work is based upon and directly complements the recent paper  \cite{JostZhang21}, 
where the authors provide preliminary results on nonlinear  spectral duality. Although the  theorems in \cite{JostZhang21} work for norm duality and convex conjugate, no investigation on   multiplicities  and variational eigenvalues is carried out there and, moreover, they require additional positivity assumptions on the associated functions.

%  There are many concepts of duality for convex  functions, including  Fenchel's convex conjugate  (i.e. Legendre transform), norm duality, and the  polarity transform (or  $\A$-transform) introduced and investigated in \cite{AM11,AR17}.  There is also a wide literature on nonlinear eigenvalue problems for discrete structures,  including spectral theory for graph $p$-Laplacians,  matrix  and tensor norms. In a recent paper  \cite{JostZhang21}, 
% the authors provide some  preliminary results on  spectral duality. Although the  theorems in \cite{JostZhang21} work for norm duality and convex conjugate, they do not derive any results on   multiplicities  and variational eigenvalues, and moreover, they require additional positivity assumption on the associated functions. In this paper, we shall systematically study  duality properties for spectra of discrete structures (including data on eigenvalue multiplicities  and variational eigenvalues).  Many kinds  of duality for  functions with  weak conditions will be considered. 

The rest of the paper is structured as follows: In Section \ref{sec:spectrum_fg} we introduce the class of functions of interest and the associated notions of spectrum and variational spectrum. In Section \ref{sec:norm_dual} we introduce the notion of norm-like dual for the class of one-homogeneous functions of interest and we review several preliminary properties for this duality operator, Then, in Section \ref{sec:main} we present our main result, showing the spectral invariance for one-homogeneous functions under norm-like duality. In Section \ref{sec:AL} we then move on to the class of $p$-homogeneous functions, for $p\geq 1$. We introduce the Legendre and polarity duality mappings and we extend  the nonlinear spectral duality theorem to these two alternative notions of duality. Finally, in Section \ref{sec:app} we illustrate a number of example problems from graph theory, network science, and convex geometry, where the new spectral duality theory can be used to provide new insight.

\subsection{Notation}
We  deal with  real finite-dimensional spaces, thus we will equivalently write $\vec x^\top \vec y$ or $\langle \vec x, \vec y\rangle$ to denote the Euclidean scalar product. For an operator (or a function) $\f$, we let   $f^{-1}(\vec y)=\{\vec x:f(\vec x)=\vec y\}$ denote the preimage of $f$ at $\vec y$ and we equivalently write $\Ker( \f)$ and $\f^{-1}(\vec 0)$ to denote the set $\{\vec x : \f(\vec x) = \vec 0\}$. 
 We do not differentiate between %don't distinguish 
 a matrix $A\in\R^{n\times m}$ and the corresponding linear map $A:\R^m\to\R^n$. For a set $S$, we write  $\mathrm{cone}(S):=\{\lambda \vec v \, : \, \lambda >0,\vec v\in S\}$.

\section{Convex homogeneous functions and their spectrum }\label{sec:spectrum_fg}
Consider two real valued functions $\f,\g:\R^n\to \R$ and suppose they are differentiable. The critical points and critical values of the ratio $r(\vec x) = \f(\vec x)/\g(\vec x)$, i.e.\ the pairs $(\lambda,\vec x^*)$ such that $\nabla r(\vec x^*)=0$ and $r(\vec x^*)=\lambda$,  define what we call (nonlinear) spectrum of the function pair $(\f,\g)$. This is because, $\nabla r(\vec x^*)=0$ if and only if $\vec x^*$ is such that 
$$
\nabla \f (\vec x^*) = \lambda \nabla \g(\vec x^*)\, .
$$
This definition still makes sense without the differentiability assumption. In that case, we can consider  Clarke's  sub-differential $\partial$ to show that if $0 \in \partial r(\vec x^*)$ then $0\in \partial \f(\vec x^*)-\lambda \partial \g(\vec x^*)$. However, the reverse implication is in general not true without assuming the functions to be differentiable. Overall, we have

\begin{definition}
Given $\f,\g:\R^n \to \R$, we call $(\lambda,\vec x)$ an eigenpair for the function pair $(f,g)$ if 
$$
0 \in \partial \f(\vec x) - \lambda \partial \g(\vec x) 
$$
where $\partial$ denotes Clarke's generalized derivative \cite{Clarke}.
\end{definition}

In the linear setting, eigenvectors are defined up to scale. The same fundamental property holds when $f$ and $g$ are homogeneous functions. 
Recall that a  function $f:\R^n \to \R$ is (positively) $p$-homogeneous if $f(\lambda \vec x)= \lambda^p f(\vec x)$ for all $\vec x \in \R^n$ and all $\lambda \in \mathbb R$, $\lambda >0$. We call $p$ the homogeneity degree of $f$. For the special cases  $p=1$ and  $p=0$ we equivalently say that $f$ is one-homogeneous and scale-invariant, respectively. In particular, in this work, we will focus on the class of homogeneous functions that are convex and have a linear kernel. This type of functions appears frequently in a large number of applications, some of which are discussed in Sections \ref{sec:introduction} and \ref{sec:app}. Precisely, we define
\begin{definition}\label{def:CH_p}
For $p\geq 1$,  let $\CH_p(\R^n)$ denote the collection of all positively $p$-homogeneous functions $\f:\R^n\to \R$ with the following properties:
\begin{enumerate}
\item $f$ is convex and nonnegative, i.e.\ $f(\vec x) \geq 0$ for all $\vec x \in \R^n$;
\item $\Ker(f)$ is a linear subspace of $\R^n$.
\end{enumerate}
\end{definition}
We remark that properties 1 and 2 above imply that any $f\in \CH_p(\R^n)$ is such that $\f(\vec x+\vec z)=\f(\vec x)$ for any $\vec z\in \Ker(f)$ and $\vec x\in\R^n$. A possible proof of this property is as follows.
Assume the contrary holds:  $f(\vec x+\vec z)>f(\vec x)$ for some $\vec x\in\R^n$ and  $\vec z\in \Ker (f)$. Fix such $\vec x$ and $\vec z$, and let $\delta=f(\vec x+\vec z)-f(\vec x)>0$. By the convexity  of $f$, for any $t\ge0$, $\frac{1}{1+t}f(\vec x+(1+t)\vec z)+\frac{t}{1+t}f(\vec x)\ge f(\vec x+\vec z)$,  which is equivalent to 
\begin{equation}\label{eq:by-convexity-f}
f(\vec x+(1+t)\vec z)\ge f(\vec x+\vec z)+t(f(\vec x+\vec z)-f(\vec x))=f(\vec x+\vec z)+t\delta.
\end{equation}
Since $\Ker(f)$ is a vector space, $\vec z\in\Ker (f)$ implies  $(1+t)\vec z\in\Ker (f)$.  Then, it follows from $(1+t)\vec z\in\Ker (f)$ and the convexity and $p$-homogeneity of $f$ that 
 $$\frac12f(\vec x)=\frac{f(\vec x)+f((1+t)\vec z)}{2}\ge f\left(\frac{\vec x+(1+t)\vec z}{2}\right)=\frac1{2^p} f(\vec x+(1+t)\vec z)$$ which yields $2^{p-1}f(\vec x)>f(\vec x+(1+t)\vec z)$. Together with \eqref{eq:by-convexity-f}, we obtain $2^{p-1}f(\vec x)>f(\vec x+\vec z)+t\delta$ for any $t>0$, but it is impossible, because the right-hand-side tends to $+\infty$ when we take $t\to+\infty$.

In general, there can be infinitely many eigenvalues for a function pair, unless $\f$ and $\g$ are quadratic, in which case the corresponding eigenpairs are standard linear eigenvalue problems.
One remarkable properties of the spectrum of homogeneous function pairs is that, when $\f$ and $\g$ are homogeneous with the same homogeneity degree  and $g(\R^n\setminus \vec 0)\subseteq \R\setminus \vec 0$, similarly to the linear eigenvalue problem case, we can identify a set of $n$ variational eigenvalues for the function pair $(f,g)$ via the Lusternik-Schnirelmann theory. In fact, in that case  the ratio $r(\vec x) = f(\vec x)/g(\vec x):\R^n\setminus \vec 0 \to\R$ is scale invariant and one has that $\vec 0 \in \partial r(\vec x)$ implies that the pair $(r(\vec x),\vec x)$ is an eigenpair for $(f,g)$. 
Hence, 
a set of $n$ eigenvalues for  $(f,g)$ can be identified via the following variational characterization: 
\begin{equation}\label{eq:variational_eigs}
   \lambda_k =  \lambda_k(f,g)= \inf_{\substack{\mathrm{genus}(S)\ge k\\ S\subset \R^n\setminus\vec0}}\; \sup\limits_{\vec x\in S}\; r(\vec x),\qquad  k=1,\cdots,n,
\end{equation}
where $\mathrm{genus}(S)$ denotes the Krasnoselski's genus of the closed, symmetric set $S$ (see e.g.\ \cite{Krasnoselski}), whose precise defintion is recalled below.

\begin{definition}[Krasnoselksii genus]\label{def:krasno}
  Let $\mathcal A$ be the class of
  closed symmetric subsets of $\R^n$,
  $\mathcal A=\lbrace S \subseteq \R^n :   S \text{ closed, }
  S=-S\rbrace\; .$
  For any $S\in\mathcal A$, let $C_k(S)=\{\varphi:S\to \mathbb R^k\setminus\{0\}, \text{continuous, s.t. $\varphi(x)=-\varphi(-x)$}\}$.  The Krasnoselskii genus of $S$ is the number defined as
  \begin{equation*}
    \mathrm{genus}(S)=
    \begin{cases}
      \inf\lbrace k\in\mathbb{N}\,: \,
      \exists\,\varphi\in C_k(S) \rbrace &
       \\
   \infty, \text{ if there exists no such $k$}
       \\
   0,   \text{ if $A=\emptyset$}
    \end{cases}\, .
  \end{equation*}
\end{definition}

This definition of variational eigenvalues \eqref{eq:variational_eigs} is a generalization of the Courant-Fisher min-max characterization of the  eigenvalues $Ax = \lambda Bx$ of the  pair of symmetric matrices $(A,B)$. In fact,  the Krasnoselski genus is a homeomorphism-invariant generalization to symmetric sets of the notion of dimension. In particular, $\mathrm{genus}(S)\geq k$ for any linear subspace $S\subseteq \R^n$ of dimension greater than $k$. Thus,   Courant-Fisher's characterization is retrieved from \eqref{eq:variational_eigs} when $S$ is any linear subspace, the genus is replaced by the dimension of $S$ and $(f,g)$ are the quadratic functions $f(x) = x^\top A x$ and $g(x) = x^\top B x$. 
In particular, note that  $\lambda_n(f,g)=\max_{x\neq 0}r(x)$, $\lambda_1(f,g) = \min_{x}r(x)$  and that, since $f^{-1}(0)$ is linear, the smallest nonzero eigenevalue of $(f,g)$ always coincides with the smallest nonzero variational eigenvalue, i.e.,
\[\lambda_{d_f+1}(f,g) = \min\{\lambda \text{ eigenvalue of }(f,g) : \lambda >0 \}
\]
where $d_f = \dim f^{-1}(0)$.

\begin{remark}[On the use of the Lusternik-Schnirelmann category index]
The Krasnoselski's genus is arguably the most popular index function in the context of variational eigenvalues for nonlinear function pairs. However, when $r$ is not even, this index cannot be used and other set measures may be required. One possibility is to use the original Lusternik-Schnirelmann category index $\mathrm{cat}(S)$ \cite{LS34,Ballmann,CLP03,FMV15}. However, since $\R^n\setminus\{\vec0\}$ is homotopy equivalent to $\mathbb{S}^{n-1}$ and $r=f/g$ is zero-homogeneous on $\R^n\setminus\{\vec0\}$, it follows from $\mathrm{cat}(\mathbb{S}^{n-1})=2$ that the 
original Lusternik-Schnirelmann category can only characterize the minimum and maximum eigenvalues in general (in contrast,  $\mathrm{genus}(\mathbb{S}^{n-1})=n$ means that the genus can be used to characterize $n$ variational  eigenvalues when $r$ is even). To characterize more variational  eigenvalues for $r$ not even, we need to add further assumptions on $r$. For example, if   $r:(\R^{n_1}\setminus\{\vec0\})\times\cdots\times(\R^{n_m}\setminus\{\vec0\})\to\R$ is a locally Lipschitz function which is zero-homogeneous on each component, that is, $r(t_1\vec x^1,\cdots,t_m\vec x^m)=r(\vec x^1,\cdots,\vec x^m)$ for any $t_i>0$ and  $\vec x^i\in\R^{n_i}$,  $i=1,\cdots,m$, we may use the Lusternik-Schnirelmann category to define $m+1$ eigenvalues of $(f,g)$, as   $(\R^{n_1}\setminus\{\vec0\})\times\cdots\times(\R^{n_m}\setminus\{\vec0\})$ is homotopy equivalent to $\mathbb{S}^{n_1-1}\times\cdots\times\mathbb{S}^{n_m-1}$ whose category is $m+1$. We emphasize that all the theorems of this paper hold unchanged if $\mathrm{genus}$ is replaced by $\mathrm{cat}$. We omit the required straightforward adjustments to the corresponding proofs for the sake of brevity. 
\end{remark}

In the next sections, we will consider three notions of duality transforms for functions in $\CH_p(\R^n)$: the norm duality, the 
Legendre transform and the polarity transform \cite{AM11,AR17}.
To ensure that the class of functions $\CH_p(\R^n)$ is closed under such transforms, we make a small modification to these dual operations by composing them with  the orthogonal projection onto $\Ker(f)^\bot$, as we will detail later. 
If one wants to study  classes of convex and homogeneous functions where $\Ker(f)$ can be nonlinear and can take the value $+\infty$, one should instead use the standard versions  of these dual  operations. 
It is quite interesting that most of the results we present in this paper still hold in a certain  sense if we use the standard versions of   the three transforms, as we will briefly discuss in Section \ref{sec:comparison}.

\section{Norm-like duality} \label{sec:norm_dual}
Any norm $\|\cdot \|$ on $\R^n$ is a convex, one-homogeneous, nonnegative function and admits a duality transform by means of which one defines the dual norm  $\|\vec x\|^* := \sup\{\langle \vec y, \vec x\rangle :\|\vec y\|\leq 1\}$. The dual norm inherits many properties from the original norm $\|\cdot\|$ and moving from one norm to the other can be of help in many applications. For a review of properties, we refer to 
\cite{Boyd,Clarke,Rockafellar,Yosida,Zeidler}. 
A similar dual operator $\D$ can be defined for general nonnegative one-homogeneous convex functions in $\CH_1(\R^n)$, as we discuss below. Our main result %of this section 
shows that the considered norm-like duality transform preserves the eigenpairs of any nonnegative homogeneous function pair in $\CH_1(\R^n)$, as well as the corresponding multiplicities, and their variational eigenvalues.%the multiplicities of their variational eigenvalues. 

On $\CH_1(\R^n)$, consider the dual operator   $\D:\CH_1(\R^n)\to \CH_1(\R^n)$  defined by 
$$
\D \f(\vec x):=\sup \Big\{ \langle \vec y, \vec x\rangle : \f(\vec y)\leq 1  \text{ and } \vec y\bot \Ker(\f) \Big\}
$$ 
for any $\f\in \CH_1(\R^n)$. It is worth noting that  one should be careful with the notation above, as $\D \f(x)$ denotes  the dual of $f$ at  $\vec x$, which  implicitly depends on the variable $\vec x$ itself. 

Note that this dual operator is essentially a composition of the ``standard'' norm dual operator $\f^*(\vec x) = \sup\{\langle \vec y, \vec x\rangle : \f(\vec y)\leq 1\}$  and a projection onto the orthogonal complement of $\Ker(f)$. In other words, if $P$ denotes the orthogonal projection onto $\Ker(f)$, then it is easy to see that it holds 
\begin{equation}\label{eq:norm-dual-identity}
\D \f(\vec x) = \f^*(\vec x - P\vec x)\, .    
\end{equation}
We use the ``modified'' dual $\D f$ instead of the standard norm dual $f^*$ because we want to work on the function space $\CH_1(\R^n)$ and we want  $\CH_1(\R^n)$  to be closed under the  dual operation. However,  $f^*(\vec x)=+\infty$ for $\vec x\not\in (\Ker f)^\bot$ and $f\in\CH_1(\R^n)$. 

A number of useful properties follow directly from the above definition of $\D$, we discuss some of them in the following.

\begin{proposition}\label{pro:basic-D}
For any $\f\in \CH_1(\R^n)$ it holds  
 $\Ker(f)=\Ker(\D f)$, 
 $\D\D \f=\f$ and, in particular, $\D \f \in \CH_1(\R^n)$. 
 \end{proposition}
\begin{proof}
Clearly, $\vec x\not\in \f^{-1}(\vec 0)$ if and only if there exists $\vec y\bot \f^{-1}(\vec 0)$ such that $\langle\vec y,\vec x\rangle>0$. This means that $\f(\vec x)>0\Leftrightarrow \D \f(\vec x)>0$ for any given $\vec x$, which implies  $\f^{-1}(\vec 0)=(\D \f)^{-1}( 0)$. 

By \eqref{eq:norm-dual-identity}, $\D\D \f(\vec x)=\D \f^*(\vec x - P\vec x)=\f^{**}(\vec x - P\vec x-P(\vec x - P\vec x))=\f^{**}(\vec x - P\vec x)=\f(\vec x - P\vec x)=\f(\vec x)$ for any $\vec x\in\R^n$, where we used the well-known identity $\f^{**}=\f$. So, $\D\D \f=\f$. 

For any $\vec z\in (\D \f)^{-1}( 0)=\f^{-1}( 0)$, $P\vec z=\vec z$, and  $\D \f(\vec x+\vec z)=\f^*(\vec x+\vec z - P(\vec x+\vec z))=\f^*(\vec x - P\vec x)=\D \f(\vec x)$. Therefore, $\D \f \in \CH(\R^n)$. 
\qed
\qed\end{proof}

\begin{proposition}
Let $S\subseteq\R^n$ be a bounded set and let $\mathrm{conv}(S)$ be its convex hull. Suppose $\vec 0$ is in the relative  interior of $\mathrm{conv}(S)$, and consider the support function   $\f_S(\vec x):=\sup_{\vec v\in S}\vec x^\top \vec v$.  Then 
$$
\D \f_S(\vec x)=\inf\Big\{\sum_i \alpha_i: f_S\Big(\sum \alpha_i \vec v_i-\vec x\Big)=0\text{ where }\alpha_i\geq0,\vec v_i\in S\Big\}.
$$
 \end{proposition}
 \begin{proof}
 Note that if $\f(\vec y)\le 1$ and $\vec v\in S$, then $\langle\vec y,\vec v\rangle\le1$. Hence, $\D \f(\vec v)\le 1$, $\forall\vec v\in S$. Let
 $$\mathcal{F}=\{\f'\in CH^+(\R^n):\f'^{-1}(0)=\f^{-1}(0)\text{ and }\f'(\vec v)\le 1,\,\forall\vec v\in S\}.$$
Then, by Proposition \ref{pro:basic-D}, we obtain  $\D \f\in \mathcal{F}$. 
 For any $\f'\in \mathcal{F}$, it is clear that  $S\subset (\f^{-1}(0))^\bot=(\f'^{-1}(0))^\bot=(\D \f')^{-1}(0)^\bot$, and thus  $$\f(\vec y)=\sup\limits_{\vec v\in S}\langle\vec y,\vec v\rangle\le \sup\limits_{\vec v\bot \f'^{-1}(0):\f'(\vec v)\le 1}\langle\vec y,\vec v\rangle=\D \f'(\vec y)$$
 which implies  $\D \f\ge  \f'$. That is, $\D \f$ is the largest function in $\mathcal{F}$. 
 
Consider the function $\widetilde{\f}:\vec x\mapsto \inf\{\sum \alpha_i:\f(\sum \alpha_i \vec v_i-\vec x)=0\text{ for some }\alpha_i\ge0,\vec v_i\in S\}$. Clearly, $\widetilde{\f}\in CH^+(\R^n)$,   $\widetilde{\f}(\vec v)\le1$, $\forall\vec v\in S$, and $\widetilde{\f}^{-1}(0)=\f^{-1}(0)$. That is, $\widetilde{\f}\in\mathcal{F}$.  

 For any $\f'\in \mathcal{F}$,  
 $\f'(\vec x)\le \sum \alpha_i\f'(\vec v_i)\le \sum \alpha_i$ whenever $\vec x-\sum \alpha_i\vec v_i\in \f^{-1}(0)$.  Taking the infimum, we get $\f'(\vec x)\le \widetilde{\f}(\vec x)$. In consequence,  we have proved that $\widetilde{\f}$ is also the largest function in $\mathcal{F}$. The proof of $\D \f=\widetilde{\f}$ is then  completed.
 \qed\end{proof}
 Let $f_S$ be defined as in the proposition above. Clearly one has $f_S(\vec x)= \sup_{v\in \mathrm{conv}(S)}\langle\vec x,\vec v\rangle$, thus we may assume without loss of generality that $S$ is convex. In that case, if we assume $S$ centrally symmetric, then $f_S$ defines a semi-norm and 
 $$
 \D \f_S(\vec x)=\inf\left\{\sum |\alpha_i|:\vec x-\sum \alpha_i \vec v_i\bot  \mathrm{span}(S)\text{ where }\vec v_i\in S\right\}\, .
 $$
 In addition, given a norm $\|\cdot\|$ and a  subset $S\subset \{\vec v:\|\vec v\|=1\}$ with $ \mathrm{conv}((-S)\cup S)=\{\vec v:\|\vec v\|\le 1\}$, we have 
$\|\vec x\|=\inf\left\{\sum |\alpha_i|:\sum \alpha_i \vec v_i=\vec x,\vec v_i\in S\right\}$. For example, we can take $S$ as the set of the extreme points of the unit ball $\{\vec v:\|\vec v\|\le1\}$, and this implies the known identity $\|A\|_{l^2\to l^2}=\inf\{\sum|\alpha_i|:A=\sum \alpha_i U_i\text{ with }U_i\text{ unitary}\}$, for a square matrix $A$.

Finally, we remark that, given a  norm $\|\cdot\|$ on $\R^n$ and a  linear subspace $X$ of $\R^n$, the map $\vec x\mapsto \inf\{\|\vec z\| : {z-x\bot X} \}$ defines a semi-norm on $\R^n$. In other terms, 
$
[\vec x]\mapsto \inf\{\|\vec y\|: {y-x\in X}\}
$ 
defines a norm on the quotient space $\R^n/X$ (we refer to Gromov's norm for this basic construction  \cite{Gromov}).

 \subsection{Linear transformation of homogeneous functions}

Given a matrix $A\in \R^{m\times n}$, i.e.\ a linear map from $\R^n$ to $\R^m$, let  $\M_A:\CH_1(\R^n)\to \CH_1(\R^m)$ be defined as 
$$
\M_A \f(\vec x):= \f(A^\top\vec x),\;\;\forall \f\in \CH_1(\R^n),\,\forall\vec x\in\R^m
$$
where $A^\top$ denotes the transpose of $A$. Let $P_A$ denote the orthogonal projection onto $\Ran A$. As $\R^m = \Ran A \oplus \Ker{A^\top}$ we can uniquely define the operator 
$\PA_A:\CH_1(\R^n)\to \CH_1(\R^m)$ as the composition of the  so-called {infimal postcomposition} $A\triangleright f$ (see e.g.~\cite{BauschkeCombettes}) and the orthogonal projection $P_A$. Precisely, we set
$$\PA_A \f(\vec x):=A\triangleright f(P_A\vec x)$$
where 
$$
A\triangleright f (\vec x) =  \inf_{\vec y : A\vec y=\vec x} f(\vec y)\, .
$$
We use this  slightly modified version of the infimal postcomposition because $A\triangleright f(\vec x)=+\infty$ for $\vec x\not\in\Ran A$.

\begin{proposition}\label{thm:double-dual}
Given $\f\in \CH_1(\R^n)$, if $\Ker\f  \subseteq \Ker A$, then $\PA_A = \D \M_A\D$. In particular, if $f$ is positive  (i.e.\ $\f(\vec x)>0$ whenever $\vec x\neq \vec 0$) then $\PA_A = \D \M_A\D$ holds for any matrix $A$.
\end{proposition}

\begin{proof}
Keeping the assumption $\f^{-1}(0)\subset\mathrm{Ker}(A)$ in mind, we have
\begin{align*}
\PA_A\D \f(\vec x)&=\inf\limits_{\vec y\in A^{-1}(\vec x)}\sup\limits_{\vec u\bot \f^{-1}(0):\f(\vec u)\le 1}\langle \vec y,\vec u\rangle = \sup\limits_{\vec u\bot \f^{-1}(0):\f(\vec u)\le 1}\inf\limits_{\vec y\in A^{-1}(\vec x)}\langle \vec y,\vec u\rangle \\
&=\sup\limits_{\vec u\bot \mathrm{Ker}(A):f(\vec u)\le 1}\langle \vec y,\vec u\rangle =\sup\limits_{\vec v:\f(A^\top\vec v)\le 1}\langle \vec y,A^\top\vec v\rangle  \\
&=\sup\limits_{\vec v\bot\mathrm{Ker}(A^\top):\f(A^\top\vec v)\le 1}\langle A\vec y,\vec v\rangle =\sup\limits_{\vec v\bot (\f\circ A^\top)^{-1}(0):\f(A^\top\vec v)\le 1}\langle \vec x,\vec v\rangle 
\\&=\D \M_A \f (\vec x) .
\end{align*}
In the above equalities, we should note that the condition $\f^{-1}(0)\subset\mathrm{Ker}(A)$ implies  $\mathrm{Ker}(A^\top)=(\f\circ A^\top)^{-1}(0)$. In fact, $A^\top\vec z=0$ $\Rightarrow$ $\f(A^\top\vec z)=0$ $\Rightarrow$ $AA^\top\vec z=0$  $\Rightarrow$ $A^\top\vec z=0$  which means $A^\top\vec z=0$ $\Leftrightarrow$ $\f(A^\top\vec z)=0$. Then, the second equality from below is proved.

Replacing $\f$ by $\D \f$, we have $\PA_A \f=\PA_A \D \D \f=\D \M_A \D \f$.
\qed\end{proof}

Before moving on, we collect in the next remark an interesting geometric interpretation of $\D$, $\M_A$ and $\PA_A$.
\begin{remark}
Consider a convex body $K$ in $\R^n$, it is well-known that the Minkowski functional of $K$ equals the support function of its dual convex body $K^\circ$. The dual operator transforms the Minkowski functional of $K$ to its support function, while $\PA_A$ maps  the Minkowski functional of $K$ to the Minkowski functional of $A(K)\times \mathrm{Ker}(A)$, and  $\M_A$ maps  the support function of $K$ to the support function of $A(K)$. If $A$ is further assumed to be a projection, then $\M_A$ maps  the Minkowski functional of $K$ to the   Minkowski functional of $K\cap \mathrm{Ker}(A)^\bot$, while $\PA_A$ transforms   the support function of $K$ to the support function of $K\cap \mathrm{Ker}(A)^\bot$.
\end{remark}

Note that, as a consequence of Proposition  \ref{thm:double-dual},  if $\f\in\CH_1(\R^n)$ is positive,  $n=m$ and $A$ is an invertible matrix, we have $\D \M_A \D \f(\vec x)=\f(A^{-1}\vec x)$, and therefore, $\D \M_A \D \f(\vec x) =  \M_A \f(\vec x)$  whenever $A$ is an orthogonal matrix. 
Moreover, for a general $\f\in\CH_1(\R^n)$, we have the identities $\M_A\D \f=\D \PA_A \f$ and $\D \M_A \f=\PA_A \D \f$. The equality $\M_A\D \f=\D \PA_A\f$   means that  ``the section of the dual equals the dual of the projection'', which is a useful observation with direct implications in convex geometry. 
On the other hand, the equality $\D\M_A \f=\PA_A\D \f$ has a similar geometrical meaning, and it has an interesting additional consequence, which we summarize in the following proposition.

\begin{proposition}\label{pro:compo-norm}
Let $\|\cdot\|$ be a monotonic norm on $\R^d$, i.e., $\|(t_1,\cdots,t_d)\|=\|(|t_1|,\cdots,|t_d|)\|$ for any $(t_1,\cdots,t_d)\in\R^d$. Let $\g_i\in \CH_1(\R^{n_i})$ be positive-definite,  and let $A_i:\R^n\to \R^{n_i}$ be a linear map, i.e., $A_i\in\R^{n_i\times n}$,  $i=1,\cdots,d$. Denote by $\hat{\g}(\vec x)=\|(\g_1(A_1\vec x),\cdots,\g_d(A_d\vec x))\|$. Then 
$$\D\hat{\g}(\vec x)=\inf\limits_{\sum_{i=1}^dA_i^\top \vec x_i=\vec x}\|(\D \g_1 (\vec x_1),\cdots,\D \g_d(\vec x_d))\|_*,$$
where $\|\cdot\|_*$ is the dual norm induced by $\|\cdot\|$.
\end{proposition}
Note that, by letting $\g_1,\cdots,\g_d$ be norms, we immediately obtain Theorem 6 in \cite{GHT20}, which has implications in the design of converging iterations for general matrix norm computations.
\begin{proof}
Let $\tilde{\g}(\vec x_1,\cdots,\vec x_d)=\|(\g_1(\vec x_1),\cdots,\g_d(\vec x_d))\|$, $\forall(\vec x_1,\cdots,\vec x_d)\in\R^{n_1}\times\cdots\times\R^{n_d}$. Then,
\begin{align*}
\| (\D \g_1 (\vec x_1),\cdots,\D \g_d (\vec x_d))\|&=\sup\limits_{\|(t_1,\cdots,t_d)\|\le 1}\sum_{i=1}^d t_i\D \g_i (\vec x_i)
\\&=\sup\limits_{\|(|t_1|,\cdots,|t_d|)\|\le 1}\sum_{i=1}^d |t_i|\sup\limits_{\g_i(\vec y_i)\le 1}\langle\vec x_i,\vec y_i\rangle
\\&=\sup\limits_{\|(|t_1|,\cdots,|t_d|)\|\le 1}\sum_{i=1}^d \sup\limits_{\g_i(\vec y_i)\le |t_i|}\langle\vec x_i,\vec y_i\rangle
\\&=\sup\limits_{\|(|t_1|,\cdots,|t_d|)\|\le 1} \sup\limits_{\g_i(\vec y_i)\le |t_i|,\forall i}\sum_{i=1}^d\langle\vec x_i,\vec y_i\rangle
\\&=\sup\limits_{\|(\g_1(\vec y_1),\cdots,\g_d(\vec y_d))\|\le 1} \sum_{i=1}^d\langle\vec x_i,\vec y_i\rangle
\\&=\sup\limits_{\tilde{\g} (\vec y_1,\cdots,\vec y_d)\le 1} \langle(\vec x_1,\cdots,\vec x_d),\vec (\vec y_1,\cdots,\vec y_d)\rangle\\&=\D \tilde{\g} (\vec x_1,\cdots,\vec x_d).
\end{align*}
Note that $\hat{\g}(\vec x)=\tilde{\g}(A^\top\vec x)$, where $A:=[A_1^\top,\cdots,A_d^\top]\in \R^{n\times(n_1+\cdots+n_d)}$. 
The proof is then completed by the identity $\D\hat{\g}=\D\M_A \tilde{\g}=\PA_A\D \tilde{\g}$.
\qed\end{proof}

\section{Main results: spectral invariance for norm-like duality}\label{sec:main}
We state here our main theorem showing that nonzero eigenvalues of function pairs, as well as their  multiplicities and their variational eigenvalues \eqref{eq:variational_eigs}, are invariant under the norm-like duality and suitable combinations of $\mathcal M_A$ and $\mathcal P_A$, for any matrix $A$. The relatively long proofs of this theorem and its main corollary cover the entire section.

Throughout the remainder of this paper, the `eigenspace' of $\lambda$ with respect to the function pair $(f,g)$ is the set $S_\lambda(f,g)$ defined by 
\[
S_\lambda(f,g) = \big\{x: 0\in \partial f(x) - \lambda \partial g(x)\big\} \, .
\]
Note that when $f$ and $g$ are even functions,  $S_\lambda(f,g)$ is a symmetric set. In this case,  we define  the multiplicity of the eigenvalue $\lambda$ for  $(f,g)$  as
\[
\mathrm{mult}_{f,g}(\lambda) = \mathrm{genus}\big(S_\lambda(f,g)\big)\, .
\]

The following main spectral invariance theorem holds.
 
\begin{theorem}\label{thm:main1}
Let $f,g\in \CH_1(\R^n)$. Then
\begin{itemize}[leftmargin=2em]
\item[P1.]  The nonzero eigenvalues  of  $(\f,\g)$ and  $(\D \g, \D \f)$  coincide. 

\item[P2.] If $f$ and $g$ are even functions, then $\mathrm{mult}_{f,g}(\lambda) = \mathrm{mult}_{\D g, \D f}(\lambda)$, for any nonzero eigenvalue $\lambda$ of $(f,g)$.

\item[P3.] If $f$ and $g$ are even functions, then the variational eigenvalues  of $(\f,\g)$ and $(\D \g,\D \f)$  coincide exactly, up to reordering. Precisely, it holds 
\[
\lambda_k(\f,\g)=\lambda_{k-d_\f+d_\g}(\D \g,\D \f) ,\qquad k=d_\f-d_{\f\g}+1,\cdots,n-d_\g
\]
where $d_{\f\g}:=\dim \f^{-1}(0)\cap \g^{-1}(0)$,  $d_\f:=\dim \f^{-1}(0)$ and $d_\g:=\dim \g^{-1}(0)$.
\end{itemize}
\end{theorem}

Moreover, combining the norm-like duality operator $\D$ with $\M_A$ and $\mathcal P_A$ for a matrix $A$, we obtain the following main consequence of the theorem above.
\begin{corollary}\label{cor:main2}
Let $\f\in \CH_1(\R^m)$, $\g\in \CH_1(\R^n)$ and $A\in \R^{n\times m}$. Then, the nonzero eigenvalues of  $(\M_{A^\top}\f,\g)$, $(\D \g,\D\M_{A^\top}\f)$, $(\M_A\D  \g, \D \f)$, $( \f,\PA_A \g)$ and $(\M_{A^\top} \f,\M_{A^\top}\PA_A%\D \M_A \D 
\g)$  coincide. Moreover, if $\f$ and $\g$ are even functions, then  
the multiplicities of the nonzero eigenvalues coincide and  the nonzero  variational eigenvalues of all these function pairs coincide exactly, up to reordering.  
\end{corollary}

We subdivide the relatively long proof of the main results above into several separate parts, 
as well as a number of smaller preliminary results that are of independent interest.

First, we prove that nonzero eigenvalues are preserved under $\D$.
\begin{proof}[Proof of Theorem \ref{thm:main1} point P1] 
 For an eigenpair $(\lambda,\vec x)$ of $(\f,\g)$ with $\lambda\ne0$ and $\vec x\ne\vec0$, it is easy to see that $\f(\vec x)=0$ $\Leftrightarrow$ $\g(\vec x)=0$, and in this case, we have  $\D \f(\vec x)=0$, $\D \g(\vec x)=0$, and  $\vec 0\in \partial \D \f(\vec x)\cap \partial \D \g(\vec x)$  which implies  $\vec 0\in \partial \D \g(\vec x)- \lambda \partial \D \f(\vec x)$. 
Hence,  $(\lambda,\vec x)$ is also an eigenpair of $(\D \g,\D \f)$. In fact, from this proof, we obtain that if $\f^{-1}(0)\cap \g^{-1}(0)\ne\{\vec0\}$, then the spectra of  $(\f,\g)$ and $(\D \g,\D \f)$ are $\R$. Therefore, without loss of generality, we assume that $\f^{-1}(0)\cap \g^{-1}(0)=\{\vec0\}$,  $\g(\vec x)=1$ and $\f(\vec x)=\lambda\ne0$. Thus, there exists $\vec u\in \partial \g(\vec x)$ such that $\lambda\vec u\in \partial \f(\vec x)$. Clearly, $\vec u\ne \vec0$. 
It follows from the fact $\partial \g(\vec x)\subset (\g^{-1}(0))^\bot=((\D \g)^{-1}(0))^\bot$ that $\D \g(\vec u)\ne0$. Moreover, we have $\langle \vec u,\vec x\rangle=\g(\vec x)=1$ by Euler's identity, and $\langle \vec u,\vec x'\rangle-1=\langle \vec u,\vec x'-\vec x\rangle\le \g(\vec x')-\g(\vec x)=\g(\vec x')-1$, $\forall \vec x'\in\R^n$ by the definition of the subgradient. Accordingly, $\D \g(\vec u)=1$, and for any
$\vec u'\in \R^n$,  $\langle \vec u'-\vec u,\vec x\rangle=\langle \vec u',\vec x\rangle-1\le \D \g(\vec u') -1= \D \g(\vec u')-\D \g(\vec u)$, which implies that $\vec x\in \partial \D \g(\vec u)$. By $\f(\vec x/\lambda)=1$ and $\lambda \vec u\in \partial \f(\vec x)=\partial \f(\vec x/\lambda)$, we similarly derive that $\vec x/\lambda\in \partial \D \f(\lambda \vec u)=\partial \D \f(\vec u)$ according to the zero-homogeneity of $\partial \f$ and $\partial\D \f$. As a consequence, $\vec0=\vec x-\lambda\cdot \vec x/\lambda\subset \partial \D \g(\vec u)-\lambda\partial \D \f(\vec u)$, i.e., $(\lambda,\vec u)$ is   an eigenpair of $(\D \g,\D \f)$. The converse also holds. And since $\partial \D \f$ and $\partial \D \g$ are scaling invariant, we indeed obtain that $\forall\vec u\in \mathrm{cone}(\partial \f(\vec x))\cap  \mathrm{cone}(\partial \g(\vec x)) $, $(\lambda,\vec u)$ is an eigenpair  of $(\D \g,\D \f)$.
\qed\end{proof}

Then, we move on to studying their multiplicities. To this end, we first observe that the genus of a compact set grows under the action of the subgradient of even functions. Here and throughout, we say  a function is $C^1$-smooth if it has continuous gradient on $\R^n\setminus\{\vec0\}$.

\begin{lemma}\label{lem:size_g_S}
Let $g\in \CH_1(\R^n)$ be an even function. Then, the   Krasnoselskii genus of a compact subset $S$ is smaller than or equal to that  of the subset $\partial g(S) := \cup_{\vec x\in S}\partial \g(\vec x)$.  
\end{lemma}

\begin{proof}
The proof is based on the deformation nondecreasing property and the continuity of the Krasnoselskii genus. 
We divide the proof  into two steps:

\begin{description}
    \item[Step 1.]  % %$\f$ or 
 Suppose that   $\g$ is $C^1$-smooth on $\R^n\setminus\{\vec0\}$.
Since the vector field induced by $\partial \g:\R^n\setminus\{\vec0\}\to \R^n\setminus\{\vec0\}$ is continuous, for any compact subset $S\subset \R^n\setminus\{\vec0\}$ with $\mathrm{genus}(S)=k$, %category large than or equal to $k$, 
the map $\vec x\mapsto %\mathrm{cone}(\partial \f(\vec x))\cap \partial \g(\vec x)=
\partial \g(\vec x)$ is continuous and if $\g$ is even, then $\partial \g$ is odd, i.e., $\partial \g(-\vec x)=-\partial \g(\vec x)$, $\forall\vec x\in\R^n$.  Therefore,  by the deformation nondecreasing property,  $\partial \g(S)$ is a subset of $\R^n$ with  $\mathrm{genus}(\partial \g(S))\ge k$. 
That is, 
for a even,  convex and smooth function $\g$, we have  $\mathrm{genus}(\partial \g(S))\ge \mathrm{genus}(S)$.

\item[Step 2.]  Suppose that $\g$ is not  $C^1$-smooth on $\R^n\setminus\{\vec0\}$.
%Both $\f$ and $\g$ are not $C^1$-smooth.

In this case, we take the Moreau-Yosida approximation of $\g$, which is defined by
$$\g_\alpha(\vec x)=\inf\limits_{\vec y\in \R^n}\g(\vec y)+\frac{1}{2\alpha}\|\vec y-\vec x\|_2^2,\;\alpha>0,$$
where we use the $l^2$-norm $\|\cdot\|_2$. It is known that $\g_\alpha$ is $C^1$-smooth and convex.
%Taking a sequence of $C^1$-smooth convex one-homogeneous function $\g_n\to \g$ such that  $\lim\limits_{n\to+\infty}\partial \g_n(\vec x)\in \g(\vec x)$.
In fact, for sufficiently small $\epsilon>0$, the size of the  $\epsilon$-neighborhood of $\partial \g(S)$ equals $\mathrm{genus}(\partial \g(S))$, and for sufficiently small $\alpha$, $\partial \g_\alpha(S)$ lies in the $\epsilon$-neighborhood of $\partial \g(S)$. Therefore, $\mathrm{genus}(\partial \g(S))\ge\mathrm{genus}(\partial \g_\alpha(S)) $,  which is  %$\mathrm{genus}(\partial \g(S))$ 
larger than or equal to $\mathrm{genus}(S)$ by Step 1.  \qed
\end{description}
\end{proof}

Next, we show that for smooth functions the subgradient maps the eigenspace of $\lambda$ as an eigenvalue of $(f,g)$ into the eigenspace of $\lambda$ as an  eigenvalue of the dual pair $(\D g, \D f)$.
\begin{lemma}\label{lem:cd}
Let $f,g\in \CH_1(\R^n)$ and let $\lambda$ be an eigenvalue of $(f,g)$. %, define 
If $\g$ is differentiable on $\R^n\setminus\{\vec0\}$, then $\partial \g(S_\lambda(\f,\g))\subset S_\lambda(\D \g,\D \f)$. Similarly, if $\f$ is differentiable, then $\partial \f(S_\lambda(\f,\g))\subset S_\lambda(\D \g,\D \f)$.
\end{lemma}

\begin{proof}
By point P1 of Theorem \ref{thm:main1} we have that 
$$\varnothing\ne\bigcup\limits_{\vec x\in S_\lambda(\f,\g)}\mathrm{cone}(\partial \f(\vec x))\cap \mathrm{cone}(\partial \g(\vec x))\subset S_\lambda(\D \g,\D \f)$$
for any eigenvalue $\lambda$ of $(\f,\g)$.  If $\g$ is derivable at any eigenvector $\vec x\in S_\lambda(\f,\g)$, then $\partial \g(\vec x)\subset \mathrm{cone}(\partial \f(\vec x))\cap \partial \g(\vec x)$. %for any  $\vec x$. 
Thus, $$\partial \g(S_\lambda(\f,\g)):=\bigcup\limits_{\vec x\in S_\lambda(\f,\g)}\partial \g(\vec x)\subset S_\lambda(\D \g,\D \f).$$

The proof of $\partial \f(S_\lambda(\f,\g))\subset S_\lambda(\D \g,\D \f)$ is similar. 
\qed\end{proof}

Finally,  we need  the following two technical properties. 
\begin{lemma}\label{lem:ef}
Let $f,g\in \CH_1(\R^n)$ and let $\lambda$ be an eigenvalue of the function pair $(f,g)$. It holds
\begin{enumerate}
\item 
The map $\vec x\mapsto \mathrm{cone}(\partial \f(\vec x))\cap \partial \g(\vec x)$ is upper semi-continuous, i.e., $\forall\vec x$,  $\forall\epsilon>0$, there exists $\delta>0$ such that for any $\vec y\in \mathbb{B}_\delta(\vec x)$, 
$\mathrm{cone}(\partial \f(\vec y))\cap \partial \g(\vec y)\subset\mathbb{B}_\epsilon(\mathrm{cone}(\partial \f(\vec x))\cap \partial \g(\vec x))$, where $\mathbb{B}_\epsilon(S)$ is the $\epsilon$-neighborhood of a subset $S$.

\item For any $\vec x\in S_\lambda(\f,\g)$, and for any $\epsilon>0$, there exists an even,  $C^1$-smooth  function $\g_x\in \CH(\R^n)$ with $\g_x^{-1}(0)=\g^{-1}(0)$  and $\delta>0$ such that 
$\partial \g_x( \mathbb{B}_\delta(\vec x))\subset \mathbb{B}_\epsilon(\mathrm{cone}(\partial \f(\vec x))\cap \partial \g(\vec x))$. 
\end{enumerate}
\end{lemma}

\begin{proof}
Point 1 follows directly from the upper semi-continuity of $\partial \f$ and $\partial \g$. Let us discuss point 2.
We only need to deal with the case that $\g$ is positive-definite.  
For any $\vec v\in\partial \g(\vec x)$, $\langle\vec x,\vec v\rangle=\g(\vec x)>0$. Then, by a standard argument in linear algebra, 
there exists a positive-definite matrix $A$ such that $A\vec x=\vec v$. Then, we take $\g_x(\vec y)=\sqrt{\langle\vec x,A\vec x\rangle\cdot\langle\vec y,A\vec y\rangle}$. It is clear that $\g_x$ is smooth,   positive-definite and convex and one-homogeneous. And it is not difficult to check that $\partial \g_x(\vec x)=A\vec x=\vec v$. Now, suppose that the vector $\vec v$ lies in $\mathrm{cone}(\partial \f(\vec x))\cap \partial \g(\vec x)$. By the above discussion, we immediately obtain that $\forall\epsilon>0$, $\exists\delta>0$ such that  $\partial \g_x( \mathbb{B}_\delta(\vec x))\subset \mathbb{B}_\epsilon(\vec v) \subset \mathbb{B}_\epsilon(\mathrm{cone}(\partial \f(\vec x))\cap \partial \g(\vec x))$.
\qed\end{proof}

\begin{proof}[Proof of Theorem \ref{thm:main1} point P2]
    
From Lemmas \ref{lem:size_g_S} and \ref{lem:cd} we have $\mathrm{genus}(S_\lambda(\f,\g))\le \mathrm{genus}(S_\lambda(\D \g,\D \f))$ if  $\f$ or $\g$ is differentiable.  Conversely, if  $\D \f$ or $\D \g$ is differentiable, 
$\mathrm{genus}(S_\lambda(\f,\g))\ge \mathrm{genus}(S_\lambda(\D \g,\D \f))$. Thus, we obtain that the multiplicity of $\lambda$ as an eigenvalue of $(\f,\g)$ coincides with the multiplicity of $\lambda$ as an eigenvalue of $(\D \g,\D \f)$. Next, we prove that the same property holds without the differentiability condition.

Let $\hat{S}_\lambda(\f,\g)=S_\lambda(\f,\g)\cap \{\vec x:\|\vec x\|_2=1\}$ be the `unit sphere' of the eigenspace corresponding to $\lambda$. Then, the multiplicity of $\lambda$ coincides with $\mathrm{genus}(\hat{S}_\lambda(\f,\g))$.   
Fix an $\epsilon>0$ such that 
\[
\mathrm{genus}\; \mathbb{B}_\epsilon\!\!\left(\bigcup\limits_{\vec x\in \hat{S}_\lambda(\f,\g)}\!\!\!\!\mathrm{cone}(\partial \f(\vec x))\cap \partial \g(\vec x)\right)\!\!=\mathrm{genus}\left(\bigcup\limits_{\vec x\in \hat{S}_\lambda(\f,\g)}\!\!\!\!\mathrm{cone}(\partial \f(\vec x))\cap \partial \g(\vec x)\right).
\]
Take $\epsilon'<\frac12\epsilon$. 
Due to Lemma \ref{lem:ef}, we can consider a family of open sets  $\{\mathbb{B}_{\delta_x}(\vec x):\vec x\in \hat{S}_\lambda(\f,\g)\}$ and the corresponding smooth function family  $\{\g_{ x}:\vec x\in \hat{S}_\lambda(\f,\g)\}$  such that for any $\vec y\in \mathbb{B}_{2\delta_x}(\vec x)$, we have 
$\mathrm{cone}(\partial \f(\vec y))\cap \partial \g(\vec y)\subset \mathbb{B}_{\epsilon'}(\mathrm{cone}(\partial \f(\vec x))\cap \partial \g(\vec x))$ and $\partial \g_x( \mathbb{B}_{2\delta_x}(\vec x))\subset \mathbb{B}_{\epsilon'}(\mathrm{cone}(\partial \f(\vec x))\cap \partial \g(\vec x))$, for a sufficiently small~$\delta_x$. %where we may assume that $\delta_x<\frac{1}{100}$. 

Since $\hat{S}_\lambda(\f,\g)$ is compact and $\{\mathbb{B}_{\delta_x}(\vec x):\vec x\in \hat{S}_\lambda(\f,\g)\}$  induces an open cover of $\hat{S}_\lambda(\f,\g)$, we can take a finite subfamily $\{\mathbb{B}_{\delta_i}(\vec x_i)\}$ of $\{\mathbb{B}_{\delta_x}(\vec x):\vec x\in \hat{S}_\lambda(\f,\g)\}$  such that 
the centers $\{\vec x_i\}$ of these open balls are  distributed  centrally symmetrically in $\R^n$,  %$S_\lambda(\f,\g)\subset\R^n$, 
and $\partial \g_i( \mathbb{B}_{\delta_i}(\vec x_i))\subset \mathbb{B}_{\epsilon'}(\mathrm{cone}(\partial \f(\vec x_i))\cap \partial \g(\vec x_i))$, where we simply write $\g_{ x_i}$ as $\g_i$.  %with $\{\vec x_i\}$ including extreme points
Then, there exist  %centrally symmetrically  distributed  
partitions of unity  $\{\psi_i\}$  subordinate to the open cover $\{\mathbb{B}_{\delta_i}(\vec x_i)\}$, i.e., $\mathrm{supp}(\psi_i)\subset \mathbb{B}_{\delta_i}(\vec x_i)$, $\psi_i\ge0$, $\sum_i\psi_i=1$ and $\psi_i=\psi_{i'}$ whenever $\vec x_i=-\vec x_{i'}$. For example, we can simply take 
$$\psi_i(\vec y)=\frac{\max\{0,\delta_i-\|\vec y-\vec x_i\|_2\}}{\sum_j\max\{0,\delta_j-\|\vec y-\vec x_j\|_2\}},\;\;\forall \vec y\in\R^n.$$
Taking $\Psi(\vec x)=\sum_i\psi_i(\vec x)\partial \g_i(\vec x)$, then $\Psi$ is a  continuous map. 

Given $\vec x\in \hat{S}_\lambda(\f,\g)$, let $I(\vec x)=\{i:\vec x\in \mathbb{B}_{\delta_i}(\vec x_i)\}$ be the index set of $\vec x$. 
Note that $\psi_i(\vec x)>0$ implies $\vec x\in \mathbb{B}_{\delta_i}(\vec x_i)$, and thus it holds $\Psi(\vec x)=\sum_{i\in I(\vec x)}\psi_i(\vec x)\partial \g_i(\vec x)$ and  $\partial \g_i(\vec x)\in \mathbb{B}_{\epsilon'}(\mathrm{cone}(\partial \f(\vec x_i))\cap \partial \g(\vec x_i))$, whenever $\vec x\in \mathbb{B}_{\delta_i}(\vec x_i)$.  %$\delta'(\vec x)=\max\{\delta_i:i\in I(\vec x)\}$, 
Moreover, there exists a bijection $\tau:I(\vec x)\to I(-\vec x)$ such that $\vec x_i=-\vec x_{\tau(i)}$,  %$\partial \g_i(-\vec x)=-\partial \g_i(\vec x)$, 
which implies  and $\psi_i(\vec x)=\psi_{\tau(i)}(-\vec x)$ and $\partial \g_i(\vec x)=-\partial \g_{\tau(i)}(-\vec x)$.  %$\tilde{\g}(\vec x)=\sum_i\psi_i$
 This implies that  
\begin{align*}
\Psi(-\vec x)&=\sum_{i\in I(-\vec x)}\psi_i(-\vec x)\partial \g_i(-\vec x)=\sum_{i\in I(\vec x)}\psi_{\tau(i)}(-\vec x)\partial \g_{\tau(i)}(-\vec x)\\&=\sum_{i\in I(\vec x)}-\psi_i(\vec x)\partial \g_i(\vec x)=-\Psi(\vec x).
\end{align*} 
 
Let  $i(x)=\mathrm{argmax}\{\delta_i:i\in I(\vec x)\}$. 
Then, for any $i\in I(\vec x)$, $\vec x_i\in \mathbb{B}_{\delta_i}(\vec x) \subset\mathbb{B}_{\delta_i}(\mathbb{B}_{\delta_{i(x)}}(\vec x_{i(x)})) =\mathbb{B}_{\delta_i+\delta_{i(x)}}(\vec x_{i(x)})\subset  \mathbb{B}_{2\delta_{i(x)}}(\vec x_{i(x)})$. Thus, $\forall i\in I(\vec x)$, $\mathrm{cone}(\partial \f(\vec x_{i}))\cap \partial \g(\vec x_{i})\subset \mathbb{B}_{\epsilon'}(\mathrm{cone}(\partial \f(\vec x_{i(x)}))\cap \partial \g(\vec x_{i(x)}))$.  
Therefore, $\partial \g_i(\vec x)\in \mathbb{B}_{2\epsilon'}(\mathrm{cone}(\partial \f(\vec x_{i(x)}))\cap \partial \g(\vec x_{i(x)}))$ for any $i\in I(\vec x)$. Consequently, we have
\begin{align*}
\Psi(\vec x)&=\sum_{i\in I(\vec x)}\psi_i(\vec x)\partial \g_i(\vec x)\in \mathbb{B}_{2\epsilon'}(\mathrm{cone}(\partial \f(\vec x_{i(x)}))\cap \partial \g(\vec x_{i(x)}))
\\&\subset \mathbb{B}_\epsilon\left(\bigcup\limits_{\vec x\in \hat{S}_\lambda(\f,\g)}\mathrm{cone}(\partial \f(\vec x))\cap \partial \g(\vec x)\right)
\end{align*}
which implies that $\Psi(\hat{S}_\lambda(\f,\g))\subset \mathbb{B}_\epsilon\left(\cup_{\vec x\in \hat{S}_\lambda(\f,\g)}\mathrm{cone}(\partial \f(\vec x))\cap \partial \g(\vec x)\right)$.  Thus, 
\begin{align*}
\mathrm{genus}(\hat{S}_\lambda(\f,\g))&\le 
\mathrm{genus}(\Psi(\hat{S}_\lambda(\f,\g))
\\&\le \mathrm{genus}\; \mathbb{B}_\epsilon\left(\bigcup\limits_{\vec x\in \hat{S}_\lambda(\f,\g)}\mathrm{cone}(\partial \f(\vec x))\cap \partial \g(\vec x)\right)
\\&=\mathrm{genus}\left(\bigcup\limits_{\vec x\in \hat{S}_\lambda(\f,\g)}\mathrm{cone}(\partial \f(\vec x))\cap \partial \g(\vec x)\right)
\end{align*}
where the first inequality is due to the fact that  $\Psi$ is odd continuous,  the second inequality is based on the nondecreasing property of the genus, and the last equality follows from the continuity of the genus.   

In summary, we have proved that for any  %(non-differentiable)
 $\f,\g\in \CH_1(\R^n)$ and any $(\lambda,x)$ eigenpair of $(f,g)$ there always holds  
\[
\mathrm{genus}\big(S_\lambda(\D g,\D f)\big)=\mathrm{genus}\left(\bigcup\limits_{\vec x\in S_\lambda(\f,\g)}\mathrm{cone}(\partial \f(\vec x))\cap \partial \g(\vec x)\right)\ge\mathrm{genus}\big( S_\lambda(\f,\g)\big).
\] 
\qed \end{proof}

\begin{proof}[Proof of Theorem \ref{thm:main1} point P3] 
We easily verify  that $0=\lambda_1(\f,\g)=\cdots=\lambda_{d_\f-d_{\f\g}}(\f,\g)<\lambda_{d_\f-d_{\f\g}+1}(\f,\g)\le\cdots\le \lambda_{n-d_\g}(\f,\g)%<\lambda_{n-d_\g+1}(\f,\g)=\cdots=+\infty
$ 
and
% \begin{align*}
$0=\lambda_1(\D \g,\D \f)=\cdots=\lambda_{d_\g-d_{\f\g}}(\D \g,\D \f)<\lambda_{d_\g-d_{\f\g}+1}(\D \g,\D \f)\le\cdots\le \lambda_{n-d_\f}(\D \g,\D \f).$ %\\&<\lambda_{n-d_\f+1}(\D \g,\D \f)=\cdots=+\infty.
% \end{align*}
Without loss of generality, we may assume that $\f^{-1}(0)\cap \g^{-1}(0)=\{\vec0\}$, and in this case,  we shall prove that $\lambda_{k-d_\f+d_\g}(\D \g,\D \f)\le  \lambda_k(\f,\g)$, $k=d_\f+1,\cdots,n-d_\g$.   
For any subset $S\subset \g^{-1}(1)$ realizing $\lambda_k(\f,\g)$ with $\mathrm{genus}(S)\ge k$, i.e., a set such that  $\sup_{\vec x\in S} f(x)/g(x)=\lambda_k(\f,\g)$,  %there is an eigenvector $\vec x\in S$ such that 
we have $\lambda_k(\f,\g)\ge \f(\vec x)/\g(\vec x)=\f(\vec x)$, $\forall \vec x\in S$. 
  Let $\mathbb{S}$ be the unit sphere in the linear subspace  $\g^{-1}(0)$  centered at the origin $\vec0$. % under any given norm. 
Let $W=\partial \g(S)* \mathbb{S}$ be the geometric join of $\partial \g(S)$ and $\mathbb{S}$, i.e., $W=\{t\vec u+(1-t)\vec y:\vec u\in \partial \g(S),\vec y\in \mathbb{S},0\le t\le1\}$.  Since $\partial \g(S):=\cup_{\vec x\in S}\partial \g(\vec x)\subset (\g^{-1}(0))^\bot$ is orthogonal to the sphere  $\mathbb{S}$  in the  linear subspace $\g^{-1}(0)$, it holds  $\mathrm{genus}(W)=\mathrm{genus}(\partial \g(S))+\mathrm{genus}(\mathbb{S})$.

For any $\vec y\in W\cap (\f^{-1}(0))^\bot$, there exist $0\le t\le 1$, $\vec u\in \partial \g(S)$ and $-\vec v\in (1-t)\mathbb{S}$, such that $\vec y=t\vec u-\vec v$.  And there exists $\vec x\in S$ such that $\vec u\in \partial \g(\vec x)$. Therefore, $\vec x\in \partial \D \g(\vec u)=\partial \D \g(t\vec u)$,  $\D \g(t\vec u)=t\D \g(\vec u)=t$ and $\D \g(\vec v)=0$.  Thus,  we have
\begin{align*}
\D \f(t\vec u-\vec v)&=\sup\limits_{\vec z\bot \f^{-1}(0)}\frac{\langle t\vec u-\vec v,\vec z\rangle}{\f(\vec z)}=\sup\limits_{\vec z\ne0}\frac{\langle t\vec u-\vec v,\vec z\rangle}{\f(\vec z)}
\ge \frac{\langle t\vec u-\vec v,\vec x\rangle}{\f(\vec x)}
\\&\ge\frac{\D \g(t\vec u)-\D \g(\vec v)}{\f(\vec x)}= %\frac{t\D \g(\vec u)-\D \g(\vec v)}{\f(\vec x)}=
\frac{t}{\f(\vec x)}\ge \frac {t}{\lambda_k(\f,\g)}    
\end{align*}
and $\D \g(t\vec u-\vec v)=\D \g(t\vec u)=t$. This implies that $\D \g(t\vec u-\vec v)/\D \f(t\vec u-\vec v)\le \lambda_k(\f,\g)$. %$$\frac{\D \g(t\vec u-\vec v)}{\D \f(t\vec u-\vec v)}\le \lambda_k(\f,\g)$$
Hence $\sup_{\vec y\in W}\D \g(\vec y)/\D \f(\vec y)\le \lambda_k(\f,\g)$.
 Now, note that
 \begin{align*}
\mathrm{genus}(W\cap (\f^{-1}(0))^\bot)&\ge\mathrm{genus}(\partial \g(S))+\mathrm{genus}(\mathbb{S})-\dim \f^{-1}(0)\\&\ge \mathrm{genus}(S)+\dim \g^{-1}(0)-\dim \f^{-1}(0)\ge k+d_\g-d_\f
 \end{align*} in which we used the claim $\mathrm{genus}(\partial \g(S))\ge \mathrm{genus}(S)$. Thus, for $k=d_\f+1,\cdots,n-d_\g$ we  obtain
  $\lambda_{k+d_\g-d_\f}(\D \g,\D \f)\le  \lambda_k(\f,\g)$. Analogously, for $k'=d_\g+1,\cdots,n-d_\f$, we have $\lambda_{k'+d_\f-d_\g}(\f,\g)\le  \lambda_{k'}(\D \g,\D \f)$. 
Substituting  $k'=k+d_\g-d_\f$ into the latter inequality, we get $\lambda_{k}(\f,\g)\le  \lambda_{k+d_\g-d_\f}(\D \g,\D \f)$, and therefore, we derive $\lambda_{k+d_\g-d_\f}(\D \g,\D \f)=  \lambda_k(\f,\g)$, $k=d_\f+1,\cdots,n-d_\g$.
\qed\end{proof}

We now move on to the proof of Corollary \ref{cor:main2}. We need one more preliminary lemma.

\begin{lemma}\label{lem:g_ker}
For $g$ and $A$ as in the statement of Corollary \ref{cor:main2}, define $\g_{\mathrm{Ker}(A)}(\vec x)=\inf\limits_{\vec x'\in \vec x+\mathrm{Ker}(A)}\g(\vec x')=\M_{A^\top}\PA_A%\D \M_A \D 
\g(\vec x)$ and let 
$
S= \{\vec x\in\R^n:\g(\vec x)=\g_{\mathrm{Ker}(A)}(\vec x)\}.
$ 
Then, $\vec x$ is an eigenvector corresponding to a nonzero eigenvalue of $(\M_{A^\top}\f,\g)$ only if $\vec x\in S$.
\end{lemma}

\begin{proof} If $\vec x\not\in S$, we shall prove that $\partial \g(\vec x)\cap \mathrm{Ker}(A)^\bot=\varnothing$. Otherwise, there exists $\vec v\in \partial \g(\vec x)$ such that $\vec v\bot\mathrm{Ker}(A)$. 
Then taking $\vec y\in\vec x+\mathrm{Ker}(A)$ such that  $\g(\vec y)=\inf\limits_{\vec x'\in \vec x+\mathrm{Ker}(A)}\g(\vec x')$, we have $0>\g(\vec y)-\g(\vec x)\ge\langle\vec v,\vec y-\vec x\rangle=0$ which leads to a contradiction. Thus, we have shown that $\partial \g(\vec x)\cap \mathrm{Ker}(A)^\bot=\varnothing$. On the other hand, $\partial_x \f(A\vec x)=A^\top \partial \f(A\vec x)\subset
\mathrm{Range}(A^\top) = \mathrm{Ker}(A)^\bot$. This implies that, for any  $\lambda\ne 0$, 
$\partial_x \f(A\vec x)\cap \lambda \partial \g(\vec x)\subset  \mathrm{Ker}(A)^\bot\cap \lambda \partial \g(\vec x)=\varnothing$, which means that $\vec x$ is not an eigenvector of any  nonzero eigenvalue of $(\M_{A^\top}\f,\g)$. The proof is completed.
\qed\end{proof}

\begin{proof}[Proof of Corollary \ref{cor:main2}]
We organize   the proof as illustrated by the diagram below
\[
\xymatrix{(\M_{A^\top} \f,\M_{A^\top}\PA_A%\D \M_A \D 
\g)\ar@{<=>}[rr]^{?} & & (\M_{A^\top}\f,\g)\ar@{<=>}[rr]^{\text{Thm \ref{thm:main1}}}\ar@{<=>}[d]^{?} & & (\D \g,\D\M_{A^\top}\f)  \\  & &(\M_A\D  \g, \D \f)\ar@{<=>}[rr]^{\text{Thm \ref{thm:main1}}} & & ( \f,\D \M_A \D \g) }
\]
Here, `$\Leftrightarrow$' denotes `spectral equivalence', i.e., the thesis holds for the two nonlinear eigenvalue problems connected by `$\Leftrightarrow$'.  
 Note that  $\M_{A^\top}\f=\f\circ A\in \CH_1(\R^n)$ and $\g\in \CH_1(\R^n)$. Thus, by Theorem \ref{thm:main1}, the thesis holds for $(\D \g,\D\M_{A^\top}\f)$ and $(\M_{A^\top}\f,\g)$. The same is true for $(\M_A\D  \g, \D \f)$ and $( \f,\D \M_A \D \g)$. 
In the remainder of the proof, we will show that the two relations marked with a `$?$' hold.

We first prove that the set of nonzero eigenvalues of  $(\M_{A^\top}\f,\g)$  coincides with the set of nonzero eigenvalues of  $(\M_A\D \g,\D \f)$.
For an eigenpair $(\lambda,\vec x)$ of $(\M_{A^\top}\f,\g)$ with $\g(\vec x)=1$, we have 
$\vec 0\in \partial_x \f(A\vec x)-\lambda \partial \g(\vec x)=A^\top \partial \f(A\vec x)-\lambda \partial \g(\vec x)$. Hence, there exists $\vec u\in \partial \g(\vec x)$ such that $\lambda\vec u=A^\top \vec v$ for some $\vec v\in \partial \f(A\vec x)$.
Thus, $A\vec x/\lambda\in\partial\D \f(\vec v) $ and $\vec x\in\partial \D \g(\vec u)=\partial \D \g(A^\top \vec v/\lambda)=\partial \D \g(A^\top \vec v)$. Therefore, $A\vec x\in A\partial \D \g(A^\top \vec v)=\partial_v \D \g(A^\top \vec v)=\partial \M_A\D \g(\vec v)$, which implies $A\vec x\in \partial \M_A\D \g(\vec v)\cap \lambda\partial\D \f(\vec v)$ and $(\lambda,\vec v)$ is an eigenpair of $(\M_A\D \g,\D \f)$. Since $(\M_{A^\top}\D\D \f,\D\D \g)=(\M_{A^\top}\f,\g)$, the converse also holds. In summary, we have shown that 
\begin{equation*}\label{eq:a}
\varnothing\ne\bigcup\limits_{\vec x\in S_\lambda(\M_{A^\top}\f,\g)}\partial \f(A\vec x)\cap(A^\top)^{-1}\left(\lambda\partial \g(\vec x)\right) \subset S_\lambda(\M_A\D \g,\D \f).
\end{equation*}
Together with Lemma \ref{lem:size_g_S}, this shows  that also the multiplicity is maintained.

Next, we show that $(\M_{A^\top} \f,\M_{A^\top}\PA_A%\D \M_A \D 
\g)$ and $(\M_{A^\top}\f,\g)$ have the same nonzero eigenvalues. By the definitions of the operators $\M_{A^\top}$ and $\PA_A$, we have
\[
\M_{A^\top}\PA_A \g(\vec x)=\inf\limits_{\vec y\in A^{-1}(A\vec x)}\g(\vec y)=\inf\limits_{\vec z\in \mathrm{Ker}(A)}\g(\vec x+\vec z).
\]
Let $g_{\Ker(A)}$ and $S$ be as in Lemma \ref{lem:g_ker}. 
For any $\vec x\in S$, 
we have
\begin{align*}
\partial_x \f(A\vec x)\cap\lambda \partial_x \g(\vec x)&=A^\top\partial \f(A\vec x)\cap\lambda \partial \g(\vec x)
\\&=A^\top\partial \f(A\vec x)\cap\mathrm{Ker}(A)^\bot\cap\lambda \partial \g(\vec x)
\\&=\partial_x \f(A\vec x)\cap\lambda \partial \g_{\mathrm{Ker}(A)}(\vec x) 
\end{align*}
where we used the fact $\partial \g_{\mathrm{Ker}(A)}(\vec x) =\partial \g(\vec x)\cap \mathrm{Ker}(A)^\bot
$. In addition, for any $\vec x$, 
$$ \partial_x \f(A\vec x)\cap\lambda \partial \g_{\mathrm{Ker}(A)}(\vec x) =\partial_x \f(A\vec x_{\mathrm{ker}})\cap\lambda \partial \g_{\mathrm{Ker}(A)}(\vec x_{\mathrm{ker}})=\partial_x \f(A\vec x_{\mathrm{ker}})\cap\lambda \partial \g(\vec x_{\mathrm{ker}})$$
where $\vec x_{\mathrm{ker}}\in S\cap(\vec x+\mathrm{Ker} (A))$. 
 Hence, together with Lemma \ref{lem:g_ker} for $\lambda\ne0$, we further  obtain 
$$ \partial_x \f(A\vec x)\cap\lambda \partial_x \g(\vec x)\ne\varnothing\Longrightarrow\partial_x \f(A\vec x)\cap\lambda \partial \g_{\mathrm{Ker}(A)}(\vec x) \ne\varnothing$$
$$\partial_x \f(A\vec x)\cap\lambda \partial \g_{\mathrm{Ker}(A)}(\vec x) \ne\varnothing \Longrightarrow\partial_x \f(A\vec x_{\mathrm{ker}})\cap\lambda \partial \g(\vec x_{\mathrm{ker}})\ne\varnothing$$
implying that 
$\lambda$ is a nonzero eigenvalue of $(\M_{A^\top}\f,\g)$ if and only if $\lambda$ is a nonzero eigenvalue of $(\M_{A^\top} \f,\M_{A^\top}\PA_A%\D \M_A \D 
\g)$, with the same multiplicity.

Finally, we  need to show the variational eigenvalues are preserved. 
For any subset $S\subset \g^{-1}(1)$ realizing $\lambda_k(\M_{A^\top}\f,\g)$ with $\mathrm{genus}(S)\ge k$, we have 
$\lambda_k(\M_{A^\top}\f,\g)\ge \f(A\vec x)/\g(\vec x)=\f(A\vec x)$, $\forall \vec x\in S$. 
Let $\mathbb{S}$ be the unit sphere in the linear subspace  $\mathrm{Ker}(A^\top)$ centered at the origin $\vec0$. %under any given norm. 
  Let $\xi:\R^n\to \mathrm{Ker}(A^\top)^\bot$ be a linear  map induced by 
  $\xi(\vec x)=(A^\top)^{-1}(\vec x)\cap \mathrm{Ker}(A^\top)^\bot$. Clearly, $\xi$ is an odd continuous map. Define the geometric  join 
    $$W:=\xi(\partial \g(S)\cap \mathrm{Ker}(A)^\bot)*\mathbb{S}.$$
  
 For any $\vec y\in W$, there exist $0\le t\le 1$, $\vec u\in \xi(\partial \g(S)\cap \mathrm{Ker}(A)^\bot)$ and $-\vec v\in (1-t)\mathbb{S}$, such that $\vec y=t\vec u-\vec v$. 
 Thus, $A^\top\vec u\in \partial \g(S)\cap \mathrm{Ker}(A)^\bot$. So, there exists $\vec x\in S$ such that $A^\top\vec u\in \partial \g(\vec x)$. Therefore, $\vec x\in \partial \D \g(A^\top\vec u)=\partial \D \g(tA^\top\vec u)$,  $\D \g(tA^\top\vec u)=t\D \g(A^\top\vec u)=t$ and $\D \g(A^\top\vec v)=0$. Note that $\langle \vec u,A\vec x\rangle=\langle A^\top\vec u,\vec x\rangle=\g(\vec x)=1$, which implies  $\vec x\not\in \mathrm{Ker}(A)$.  Then,  we have
\begin{align*}
\D \f(t\vec u-\vec v)&=\sup\limits_{\vec z\ne0}\frac{\langle t\vec u-\vec v,\vec z\rangle}{\f(\vec z)}
\ge \frac{\langle t\vec u-\vec v,A\vec x\rangle}{\f(A\vec x)}= \frac{\langle tA^\top\vec u-A^\top\vec v,\vec x\rangle}{\f(A\vec x)}
\\&\ge\frac{\D \g(tA^\top\vec u)-\D \g(A^\top\vec v)}{\f(A\vec x)}= 
\frac{t}{\f(A\vec x)}\ge \frac {t}{\lambda_k(\M_{A^\top}\f,\g)}    
\end{align*}
and $\D \g(A^\top(t\vec u-\vec v))=\D \g(tA^\top\vec u)=t$. Accordingly, we obtain
$$\frac{\D \g(A^\top(t\vec u-\vec v))}{\D \f(t\vec u-\vec v)}\le \lambda_k(\M_{A^\top}\f,\g)   $$
and then
$$\sup\limits_{\vec y\in W}\frac{\D \g(A^\top\vec y)}{\D \f(\vec y)}\le \lambda_k(\M_{A^\top}\f,\g). $$
Let $d_A=\dim\mathrm{Ker}(A)$ and $d_{A^\top}=\dim\mathrm{Ker}(A^\top)$. We  estimate the Krasnoselskii  genus of $W$ as 
\begin{align*}
\mathrm{genus}(W)&=\mathrm{genus}(\xi(\partial \g(S)\cap \mathrm{Ker}(A)^\bot))+\mathrm{genus}(\mathbb{S})
\\&\ge \mathrm{genus}(\partial \g(S)\cap \mathrm{Ker}(A)^\bot)+\dim\,\mathrm{Ker}(A^\top)
\\&\ge \mathrm{genus}(\partial \g(S))-\dim\, \mathrm{Ker}(A)+\dim\,\mathrm{Ker}(A^\top)
\\&\ge \mathrm{genus}(S)-d_A+d_{A^\top}\ge k-d_A+d_{A^\top}
\end{align*}
where the first equality uses the fact that $\xi(\partial \g(S)\cap \mathrm{Ker}(A)^\bot)\subset \mathrm{Ker}(A^\top)^\bot$ and $\mathbb{S}$ is the unit sphere of  the linear subspace  $\mathrm{Ker}(A^\top)$. Therefore, we obtain that 
\begin{equation}\label{eq:tmp1}
    \lambda_{k-d_A+d_{A^\top}}(\M_A\D \g,\D \f)\le \lambda_k(\M_{A^\top}\f,\g)\, .
\end{equation}
As the converse holds by a similar argument, we conclude that the identity holds in \eqref{eq:tmp1}. % proof of $\lambda_{k-d_A+d_{A^\top}}(\M_A\D \g,\D \f)= \lambda_k(\M_{A^\top}\f,\g) $.

To conclude, we prove that $\lambda_k(\M_{A^\top}\f,\g)= \lambda_{k-d_A}(\M_{A^\top} \f,\M_{A^\top}\PA_A \g)$. 
Let again $S$ be defined as in Lemma \ref{lem:g_ker}. %Define $S=\{\vec x\in\R^n:\g(\vec x)=\g_{\mathrm{Ker}(A)}(\vec x)\}$. 
We know that $\mathrm{genus}(S)= n-\dim\mathrm{Ker}(A)$. For any $W$ with $\mathrm{genus}(W)>\dim\mathrm{Ker}(A)$,  $\mathrm{genus}(W\cap S)\ge \mathrm{genus}(W)-\dim\mathrm{Ker}(A)$. 
It is not difficult to check that
\begin{align*}
\lambda_k(\M_{A^\top}\f,\g)&= \inf_{\mathrm{genus}(W)\ge k}\sup\limits_{\vec x\in W} \frac{\f(A\vec x)}{\g(\vec x)} \ge  \inf_{\mathrm{genus}(W)\ge k}\sup\limits_{\vec x\in W\cap S} \frac{\f(A\vec x)}{\g(\vec x)} 
\\&= \inf_{\mathrm{genus}(W')\ge k-d_A,W'\subset S}\sup\limits_{\vec x\in W'} \frac{\f(A\vec x)}{\g(\vec x)} \\
&= \inf_{\mathrm{genus}(W')\ge k-d_A,W'\cap \mathrm{Ker}(A)=\varnothing}\sup\limits_{\vec x\in W'} \frac{\f(A\vec x)}{\g_{\mathrm{Ker}(A)}(\vec x)}
\\&=\lambda_{k-d_A}(\M_{A^\top} \f,\M_{A^\top}\PA_A \g).
\end{align*}
On the other hand, for any $W$ realizing $\lambda_{k-d_A}(\M_{A^\top} \f,\M_{A^\top}\PA_A \g)$, there is an eigenvector in $W$, and every  nontrivial  eigenvector lies in $S$. 
Fix such a subset $W$, consider a family  of subsets defined by $\{(W\cap S)*(r\mathbb{S})\}_{r>1}$, where $r\mathbb{S}$ is the  sphere with radius $r$ in the linear subspace  $\mathrm{Ker}(A)$  centered at the origin $\vec0$. It is easy to check that $\mathrm{genus}(W\cap S)*(r\mathbb{S})=\mathrm{genus}(W\cap S)+ \mathrm{genus}(r\mathbb{S})\ge k-d_A+d_A=k$ for sufficiently large $r$. And one can verify that $$\lim\limits_{r\to+\infty}\sup\limits_{\vec x\in (W\cap S)*(r\mathbb{S})} \frac{\f(A\vec x)}{\g(\vec x)}=\sup\limits_{\vec x\in W\cap S} \frac{\f(A\vec x)}{\g(\vec x)} $$
which implies $\lambda_k(\M_{A^\top}\f,\g)\le \lambda_{k-d_A}(\M_{A^\top} \f,\M_{A^\top}\PA_A \g)$. Consequently, the proof of  $\lambda_k(\M_{A^\top}\f,\g)= \lambda_{k-d_A}(\M_{A^\top} \f,\M_{A^\top}\PA_A \g)$ is completed and we can conclude.
\qed\end{proof}

\section{Legendre  and Polarity transforms}\label{sec:AL}
 
In this section, we use the Legendre and the Polarity transform to provide nonlinear spectral duality results for function pairs in $\CH_p$ with $p\geq 1$, and not just $p=1$. 

First, we recall the notion of the two transforms for general functions. 
The  Legendre transform   of a  function $\f:\R^n\to\R$ is defined as 
\begin{equation*}\label{eq:original-L}
\hat{\mathcal{L}} \f(\vec x)=\sup\limits_{y\in\R^n}\langle \vec x,\vec y\rangle-\f(\vec y)=\inf\{s\in\R: \langle \vec x,\vec y\rangle\le \f(\vec y)+s,\forall \vec y\in\R^n\},    
\end{equation*}
and the  Polarity transform   of a   function $\f:\R^n\to[0,+\infty)$ is defined as 
\begin{equation*}\label{eq:original-A}
\hat{\mathcal{A}} \f(\vec x)=\sup\limits_{y:\f(y)>0}\frac{\langle \vec x,\vec y\rangle-1}{\f(\vec y)}=\inf\{c\in\R:\langle \vec x,\vec y\rangle\le c\f(\vec y)+1,\forall \vec y\in\R^n\}.    
\end{equation*}

Similar to  the norm-like dual, we now consider a modified version of the two transforms that is better suited for the function family $\CH_p(\R^n)$.
Precisely, we define the  Legendre and the Polarity transforms  of a  function $\f\in \CH_p(\R^n)$ respectively as 
\begin{align*}
  &\mathcal{L} \f(\vec x)=\sup\limits_{y\bot \f^{-1}(0)}\langle \vec x,\vec y\rangle-\f(\vec y)=\inf\{s\in\R: \langle \vec x,\vec y\rangle\le \f(\vec y)+s,\forall \vec y\in (\f^{-1}(0))^\bot\}     \\
  &\mathcal{A} \f(\vec x)=\sup\limits_{y\bot \f^{-1}(0)}\frac{\langle \vec x,\vec y\rangle-1}{\f(\vec y)}=\inf\{c\in\R:\langle \vec x,\vec y\rangle\le c\f(\vec y)+1,\forall \vec y\in(\f^{-1}(0))^\bot\}.
\end{align*}

% \begin{remark}\label{rem:LfAf}
Just like the norm dual operator, we note that  $\mathcal{L}f(\vec x)=\hat{\mathcal{L}}f(\vec x-P\vec x)$ and $\mathcal{A}f(\vec x)=\hat{\mathcal{A}}f(\vec x - P\vec x)$, where $P$ denotes the orthogonal projection onto $\Ker(f)$. We emphasize that, as for the norm-like duality,  we use these modified transforms instead of the standard ones  because $\hat{\mathcal{L}}f(\vec x)=\hat{\mathcal{A}}f(\vec x)=+\infty$ for $\vec x\not\in (\Ker f)^\bot$ and $f\in\CH_p(\R^n)$.
Nonetheless, it is quite surprising that several of the results of the main theorems in this paper still hold in a certain  sense if we use the standard concepts of  infimal postcomposition,  norm dual, Legendre transform and Polarity transform,  instead of our modified versions. For the sake of clarity, we postpone this observation to the discussion in  Section~\ref{sec:comparison}.

The next two Theorems \ref{thm:L-transform} and \ref{thm:A-transform} show spectral invariance under the two duality transforms for pairs of convex $p$-homogeneous functions $\f\in \CH_p(\R^n)$ and $\g\in \CH_q(\R^n)$, with $p,q\geq 1$. Then, in Theorem \ref{thm:AL} we will present our main result of this section, which corresponds to the Legendre and polarity transforms' version of the norm-like duality Theorem \ref{thm:main1} and Corollary \ref{cor:main2} from the previous section. In particular, Theorem \ref{thm:AL}  fully characterizes the spectral duality equivalence under the action of $\mathcal L$, $\mathcal A$, $\mathcal P_A$ and $\mathcal M_A$, for a function pair $\f\in \CH_p(\R^n)$ and $\g\in \CH_q(\R^n)$, with $p,q\geq 1$.  
Here, and in the rest of the section, for a $p>1$ we let $p^*$ be its H\"older conjugate exponent $1/p+1/p^*=1$.

\begin{theorem}\label{thm:L-transform}
For any $\f\in \CH_p(\R^n)$ and $\g\in \CH_q(\R^n)$ with some  $p,q>1$, 
the nonzero eigenvalues  of  $(\f,\g)$ and  $(\mathcal{L} \g, \mathcal{L} \f)$ coincide  up to a power factor.  Precisely, for any eigenpair $(\lambda,\vec x)$ of $(\f,\g)$ with $\lambda\ne0$ and $\g(\vec x)\ne0$,    and for any $\vec u\in \mathrm{cone}(\partial \f(\vec x))\cap   \partial \g(\vec x) $, $(\lambda^{p^*-1},\vec u)$ is an eigenpair  of $(\mathcal{L} \g,\mathcal{L}\f)$. 
\end{theorem}
\begin{proof}
It is known that  $\partial \f$ is homogeneous of degree $(p-1)$, and $\partial \g$ is homogeneous of degree $(q-1)$. Since $(\lambda,\vec x)$ is an eigenpair  of $(\f,\g)$ with $\lambda\ne0$ and $\g(\vec x)\ne0$,  the inclusion relation $\vec0\in \partial \f(\vec x) - \lambda \partial \g(\vec x) $ implies that $0=\langle\vec x,\partial \f(\vec x)\rangle - \lambda \langle\vec x,\partial \g(\vec x)\rangle =pf(\vec x)-\lambda q\g(\vec x) $. Thus, $f(\vec x)>0$ and  $\lambda=pf(\vec x)/q\g(\vec x)>0$. Moreover, there exists $\vec u\in\partial \g(\vec x)$ such that $\lambda \vec u\in \partial \f(\vec x)$. which implies $ \vec u\in\mathrm{cone}(\partial \f(\vec x))\cap   \partial \g(\vec x)\ne\varnothing $. And for any $\vec u\in \mathrm{cone}(\partial \f(\vec x))\cap   \partial \g(\vec x) $, there exists $\vec v\in\partial \f(\vec x)$ and $\mu\ge0$ such that $\vec u=\mu\vec v$. If follows from $\langle\vec u,\vec x\rangle=q\g(\vec x)\ne0$ that $\vec u\ne\vec0$ and hence $\mu>0$. 
One on hand,  $\langle \lambda\mu\vec v,\vec x\rangle=\lambda\mu\langle \vec v,\vec x\rangle=\lambda\mu p\f(\vec x)$, and on the other hand, $\langle \lambda\mu\vec v,\vec x\rangle=\langle \lambda\vec u,\vec x\rangle=p\f(\vec x)\ne0$. Thus, $\lambda\mu=1$ and $\vec v=\lambda \vec u\in \partial \f(\vec x)$. Consequently, by the property of  Legendre  transform, $\vec x\in\partial \mathcal{L} \g(\vec u)$ and $\vec x\in\partial \mathcal{L} \f(\lambda \vec u)=\lambda^{p^*-1}\partial \mathcal{L} \f( \vec u)$. This implies 
$$\vec x\in \partial \mathcal{L} \g(\vec u)\cap \lambda^{p^*-1}\partial \mathcal{L} \f( \vec u)\ne\varnothing$$
 which means that $(\lambda^{p^*-1},\vec u)$ is an eigenpair  of $(\mathcal{L} \g,\mathcal{L}\f)$.
\qed\end{proof}

\begin{theorem}\label{thm:A-transform}
For $\f\in \CH_p(\R^n)$ and $\g\in \CH_q(\R^n)$, 
the nonzero eigenvalues  of  $(\f,\g)$ and  $(\A \g, \A \f)$ coincide  up to a scaling factor.  Precisely, for any eigenpair $(\lambda,\vec x)$ of $(\f,\g)$ with $\lambda\ne0$ and $\g(\vec x)\ne0$,    and for any $\vec u\in \mathrm{cone}(\partial \f(\vec x))\cap   \partial \g(\vec x) $, $(\alpha \lambda,\vec u)$ is an eigenpair  of $(\A \g,\A \f)$, with $\alpha = (\frac pq)^{p-2}\frac{(q-1)^{q-1}}{ (p-1)^{p-1}}$. %, where $\mathrm{cone}(\partial  \f(\vec x)):=\{t\vec v:t>0,\vec v\in \partial  \f(\vec x)\}$.
\end{theorem}
\begin{proof}

 Let  $(\lambda,\vec x)$ be an eigenpair of $(\f,\g)$ with $\lambda\ne0$ and $\vec x\ne\vec0$. It is easy to see that $\f(\vec x)=0$ $\Leftrightarrow$ $\g(\vec x)=0$, and in this case, we have  $\A \f(\vec x)=0$, $\A \g(\vec x)=0$, and  $\vec 0\in \partial \A \f(\vec x)\cap \partial \A \g(\vec x)$  which implies  $\vec 0\in \partial \A \g(\vec x)- \lambda \partial \A \f(\vec x)$. 
Hence,  $(\lambda,\vec x)$ is also an eigenpair of $(\A \g,\A \f)$. In fact, from this proof, we obtain that if $\f^{-1}(0)\cap \g^{-1}(0)\ne\{\vec0\}$, then the spectra of  $(\f,\g)$ and $(\A \g,\A \f)$ are $\R$. Therefore, without loss of generality, we assume that $\f^{-1}(0)\cap \g^{-1}(0)=\{\vec0\}$,  $\g(\vec x)=1$ and $\f(\vec x)=q\lambda/p\ne0$. Thus, there exists $\vec u\in \partial \g(\vec x)$ such that $\lambda \vec u\in \partial \f(\vec x)$. Clearly, $\vec u\ne \vec0$. 
It follows from the fact $\partial \g(\vec x)\subset (\g^{-1}(0))^\bot=((\A \g)^{-1}(0))^\bot$ that $\A \g(\vec u)\ne0$. Moreover, we have $\langle \vec u,\vec x\rangle=q\g(\vec x)=q$ by Euler's identity, and $\langle \vec u,\vec x'\rangle-q=\langle \vec u,\vec x'-\vec x\rangle\le \g(\vec x')-\g(\vec x)=\g(\vec x')-1$, $\forall \vec x'\in\R^n$ by the definition of the subgradient. 
So, $\langle \vec u,\vec x'\rangle\le  \g(\vec x')+q-1$, and thus $\langle \vec u,\frac{1}{q-1}\vec x'\rangle\le (q-1)^{q-1} \g(\frac{1}{q-1}\vec x')+1$. Accordingly, $\A \g(\vec u)=(q-1)^{q-1}$, and for any
$\vec u'\in \R^n$, 
\begin{align*}
\langle \vec u'-\vec u,(q-1)^{q-1}\vec x\rangle&=(q-1)^q\langle \vec u'-\vec u,(q-1)^{-1}\vec x\rangle
\\&=(q-1)^q(\langle \vec u',(q-1)^{-1}\vec x\rangle -q (q-1)^{-1})
\\&\le (q-1)^q(\A \g(\vec u')\g((q-1)^{-1}\vec x)+1 -q(q-1)^{-1})
\\&= (q-1)^q((q-1)^{-q}\A \g(\vec u')-(q-1)^{-1})
\\&= \A \g(\vec u')-(q-1)^{q-1}=\A \g(\vec u')-\A \g(\vec u)
\end{align*}
which implies that $(q-1)^{q-1}\vec x\in \partial \A \g(\vec u)$. 
We can similarly derive that $\A \f(\vec u)=\frac {1}{\lambda} (\frac{q(p-1))}{p})^{p-1}$ and  $(\frac pq)^{2-p} (p-1)^{p-1}\frac{1}{\lambda}\vec x\in\partial \A \f(\vec u)$. 
Therefore, $(\lambda (\frac pq)^{p-2}\frac{(q-1)^{q-1}}{ (p-1)^{p-1}},\vec u)$  is an eigenpair of $(\A \g,\A \f)$.
\qed\end{proof}

\begin{remark}
Although Artstein-Avidan and Rubinstein \cite{AR17} introduce a polar subdifferential map which possesses very nice properties for the polarity transform, it is still necessary to use the usual subdifferential in Theorem \ref{thm:A-transform}. 
\end{remark}

Before presenting our main and final result of this section, we need a number of relevant preliminary observations and results.  First, we show in the next Proposition \ref{pro:basic-DAL} that both Legendre and polarity transforms are directly related to the norm-like transform of Section \ref{sec:norm_dual}.  Then, in Proposition \ref{pro:1-p-q-transfer} we show how the eigenpairs of $(f,g)$ change when $f$ and $g$ are raised to some power.  These two results will allow us to work on the spectral duality for Legendre  and polarity transforms for $p$-homogeneous convex functions by means of the results previously shown for the case %setting
of norm-like duality for one-homogeneous convex functions.

\begin{proposition}\label{pro:basic-DAL}
Given $p\ge 1$ and $r\ge 1/p$, for any  nonnegative  $p$-homogeneous  function $\f$,  $\f$ is convex if and only if $\f^r$ is convex. And, if $\f$ is nonnegative $p$-homogeneous  and convex, then
\begin{equation}\label{eq:ALD}
\mathcal{L}\f=\frac{p-1}{p^{p^*}}(\D \f^{\frac 1p})^{p^*}\text{ and }\mathcal{A}\f=\frac{(p-1)^{p-1}}{p^{p}}(\D \f^{\frac 1p})^{p}.    
\end{equation}
If  
$\f_p\in \CH_p(\R^n)$ with $p\ge1$ and $\{\f_p\}$ Gamma-converges to $\f$ as $p$ tends to $1$, then  %$\mathop{\Gamma\text{-}\lim}\limits_{p\to 1^+}\f_p=\f$, 
we have
$$\D \f=\lim\limits_{p\to 1^+}(\mathcal{L}\f_p)^{\frac{1}{p^*}}=\lim\limits_{p\to 1^+}(\mathcal{A}\f_p)^{\frac{1}{p}}=\lim\limits_{p\to 1^+}\mathcal{A}\f_p.$$
\end{proposition}

\begin{proof}
The first argument is equivalent to the statement that for any nonnegative  one-homogeneous function $\f$, and $p\ge 1$, $\f$ is convex $\Leftrightarrow$ $\f^p$ is convex. To show this property, first note that  the direction that the convexity of  $f$ implies the convexity of $f^p$ is easy since $t\mapsto t^p$ is increasing and convex on $[0,\infty)$. We now show that the convexity of  $f^p$ implies the convexity of $f$. For any $x,y$ with $f(x),f(y)>0$, letting $C=tf(x)+(1-t)f(y)$, we have
\begin{align*}
\frac{f^p(tx+(1-t)y)}{C^p}&=f^p(\frac{tf(x)}{C}\frac{x}{f(x)}+\frac{(1-t)f(y)}{C}\frac{y}{f(y)})
\\&\le \frac{tf(x)}{C}f^p(\frac{x}{f(x)})+\frac{(1-t)f(y)}{C}f^p(\frac{y}{f(y)})
\\&= \frac{tf(x)}{C}+\frac{(1-t)f(y)}{C} = 1
\end{align*}
which yields $f(tx+(1-t)y)\le C$. As the case of $f(x)f(y)=0$ is straightforward, we obtain the convexity of $f$.  The equalities shown in \eqref{eq:ALD} are presented in  \cite{AM09}.
As for the final statement, note that if $\f_p$ Gamma-converges to $\f$, then $\f_p^{1/p}$ also  Gamma-converges to $\f$. And then, by the property of Gamma-convergence, 
$\D \f_p^{1/p}(\vec x)$ converges to $\D \f$ as $p$ tends to $1$. Thus, by \eqref{eq:ALD},  $$(\mathcal{L}\f_p)^{\frac{1}{p^*}}=\left(\frac{p-1}{p^{p^*}}(\D \f_p^{\frac 1p})^{p^*}\right)^{\frac{1}{p^*}}=\frac{(p-1)^{\frac{1}{p^*}}}{p}\D \f_p^{\frac1p}\to\D \f$$
and 
$$(\mathcal{A}\f_p)^{\frac{1}{p}}=\left(\frac{(p-1)^{p-1}}{p^{p}}(\D \f^{\frac 1p})^{p}\right)^{\frac{1}{p}}=\frac{(p-1)^{\frac{p-1}{p}}}{p}\D \f_p^{\frac1p}\to\D \f$$
as $p$ tends to $1$. Clearly, $\mathcal{A}\f_p\to \D \f$, $p\to 1^+$.
\qed\end{proof}

\begin{proposition}\label{pro:1-p-q-transfer}
%\textbf{Claim}. 
For $\lambda\ne0$,  $(\lambda,\vec x)$ is an eigenpair of $(\f,\g)$ if and only if $(\frac{ap\f^{p-1}(\vec x)}{ bq \g^{q-1}(\vec x)}\lambda,\vec x)$ is an eigenpair of $(a\f^p,b \g^q)$. %, where $a,b,p,q>0$.
Moreover, the eigenpairs of  $(\f,\g)$ and $(a\f^p,b \g^q)$ have a completely  equivalent one-to-one  correspondence. 
\end{proposition}

%\textbf{Proof}.  
\begin{proof}
If $(\mu,\vec x)$  is an eigenpair of $(a\f^p,b \g^q)$ where $\mu\ne0$,  then $\f(\vec x)>0$ and $\g(\vec x)>0$, and 
\begin{align*}
\vec 0&\in \partial a\f^p(\vec x)-\mu \partial b\g^q(\vec x)=ap\f^{p-1}(\vec x)\partial \f(\vec x)-\mu bq \g^{q-1}(\vec x)\partial \g(\vec x)
\\&=ap\f^{p-1}(\vec x)\left(\partial \f(\vec x)-\frac{\mu bq \g^{q-1}(\vec x)}{ap\f^{p-1}(\vec x)}\partial \g(\vec x)\right)
\end{align*}
which implies that $\frac{\mu bq \g^{q-1}(\vec x)}{ap\f^{p-1}(\vec x)}$ is an eigenvalue of $(\f,\g)$. Conversely, it is easy to see that if $(\lambda,\vec x)$ is an eigenpair of  $(\f,\g)$ with $\lambda\ne 0$, then $(\frac{ap\f^{p-1}(\vec x)}{ bq \g^{q-1}(\vec x)}\lambda,\vec x)$ is an eigenpair of $(a\f^p,b \g^q)$. 

In addition, it is clear that $(0,\vec x)$ is an eigenpair of $(\f,\g)$ if and only if $(0,\vec x)$ is an eigenpair of $(a\f^p,b \g^q)$.
\qed\end{proof}

Finally, we point out that we need to be careful with the case $p\neq q$ when dealing with multiplicities and variational eigenvalues. In that case, in fact, $r=f/g$ is not scale-invariant and   the eigenvalues of $(f,g)$ and their multiplicities have degenerate properties. Precisely,  

\begin{lemma}\label{lem:reason-pneq}
Given $\f\in \CH_p(\R^n)$ and $\g\in \CH_q(\R^n)$  with $p,q\geq 1$, for any $\lambda\ge0$ and $t>0$,  there holds  $x\in S_\lambda(\f,\g)$ if and only if $tx\in S_{t^{p-q}\lambda}(\f,\g)$. Moreover, if $p\ne q$,  $f$ and $g$ are  even, and $(\f,\g)$ has a nonzero eigenvalue, then  the function  $\lambda\mapsto\mathrm{mult}_{f,g}(\lambda)$ is constant on $(0,+\infty)$.
\end{lemma}
\begin{proof}
By the definition of the eigenspace $S_\lambda$, and the homogeneity of $\partial\f$ and $\partial\g$, we have
\begin{align*}
x\in S_\lambda(\f,\g)&\Longleftrightarrow 0 \in \partial \f(\vec x) - \lambda \partial \g(\vec x) 
\\& \Longleftrightarrow 0 \in t^{p-1}\partial \f(\vec x) - \lambda t^{p-1} \partial \g(\vec x) 
\\& \Longleftrightarrow 0 \in \partial \f(t\vec x) - \lambda t^{p-q} \partial \g(t\vec x) 
\\& \Longleftrightarrow tx\in S_{t^{p-q}\lambda}(\f,\g).
\end{align*}
If $p\ne q$ and $(\f,\g)$ has a positive eigenvalue $\hat{\lambda}>0$, then for any $t>0$,  $t^{p-q}\hat{\lambda}$ is also an eigenvalue of $(\f,\g)$, that is, all positive numbers are eigenvalues of $(\f,\g)$. 
Note that for any $t>0$, the map $\varphi_t:\R^n\to\R^n$ defined by $\varphi_t(x)=tx$ is an odd homeomorphism. Then, for any $\lambda>0$, it follows from $\varphi_t(S_\lambda(\f,\g))=S_{t^{p-q}\lambda}(\f,\g)$ and the homeomorphism-invariance of  Krasnoselskii genus that 
$$\mathrm{mult}_{f,g}(\lambda)=\mathrm{genus}(S_\lambda(\f,\g))=\mathrm{genus}(S_{t^{p-q}\lambda}(\f,\g))=\mathrm{mult}_{f,g}(t^{p-q}\lambda).$$
By the arbitrariness of $\lambda>0$ and $t>0$, the multiplicity function $\mathrm{mult}_{f,g}(\lambda)$ is independent of $\lambda>0$.
\qed\end{proof}

Thus, when $p\neq q$ the (variational) eigenvalues of $(f,g)$ change when the corresponding eigenvector is scaled and their multiplicities are constant. To overcome this issue and have a meaningful definition of variational eigenvalues also for the $p\neq q$ case, it is useful to restrict the  variational eigenvalues to suitable centrally symmetric convex surfaces.  In particular, we note that for $p=q$ we have $f(x)/g(x) = f(x/g(x)^{1/p})$ and $g(x/g(x)^{1/p})=1$ for all $x\neq 0$. Thus, we can recast \eqref{eq:variational_eigs}  as
\begin{equation}\label{eq:restricted_variationa_eig}
    \lambda_k(f,g) = \inf\limits_{\substack{\mathrm{genus}(S)\ge k\\S\subset \g^{-1}(1)}}\;\sup\limits_{\vec x\in S} \; r(x) \; = \inf\limits_{\substack{\mathrm{genus}(S)\ge k\\S\subset \g^{-1}(1)}}\;\sup\limits_{\vec x\in S} \;  f(x)\, 
\end{equation}
i.e., for $p=q$ the $k$-th variational eigenvalue  equals the $k$-th min-max critical value of $f$ restricted to the centrally symmetric convex hypersurface $\g^{-1}(1)$.

By constraining the eigenvalues to  $g^{-1}(1)$, %we obtain the following immediate consequence of Theorems \ref{thm:L-transform} and \ref{thm:A-transform} 
% \begin{proposition}\label{pro:constrsined_spectrum}
% For $X=g^{-1}(1)$ let $\widehat{\mathrm{spec}}^+(\f,\g):=\mathrm{spec}_{X}^+(\f,\g)$. The following chain of equivalences holds: 
%  $\lambda\in \widehat{\mathrm{spec}}^+(\f,\g)$  $\Longleftrightarrow$  $\lambda^{\frac{q^*}{p}} (\frac{p^*}{q})^{\frac{q^*}{p^*}-1}\in  \widehat{\mathrm{spec}}^+(\mathcal{L} \g,\mathcal{L}\f)$  $\Longleftrightarrow$ $\lambda^{\frac qp}\frac{q}{p}(q-1)^{q-1}(\frac{p^*}{q})^{\frac{q}{p^*}}\in\widehat{\mathrm{spec}}^+(\mathcal{A} \g,\mathcal{A}\f)$. In particular, if $p=q$  we have that $\widehat{\mathrm{spec}}^+(\f,\g)=\widehat{\mathrm{spec}}^+(\mathcal{A} \g,\mathcal{A}\f)$.
% \end{proposition}
%
%
%
%
% Finally, 
the next theorem provides the Legendre and polarity transforms' version of Theorem \ref{thm:main1} and Corollary \ref{cor:main2}, i.e., it presents the overall spectral duality equivalence between Frenchel duality,  polarity transform and linear transformations.

\begin{theorem}\label{thm:AL}
For any  %positive-definite 
$\f\in \CH_p(\R^m)$, $\g\in \CH_q(\R^n)$, and linear map $A:\R^n\to\R^m$, the  strong equivalence relations illustrated in the following diagram hold: %eigenvalue problems are  strongly  equivalent up to some  scaling or power factors:
$$
\!\!\!\xymatrix{
(\mathcal{L}\M_{A^\top}\PA_A\g,\mathcal{L}\M_{A^\top}\f)
\ar@{<=>}[r]
&(\M_{A^\top}\f,\M_{A^\top}\PA_A\g)\ar@{<=>}[d]\ar@{<=>}[r]
&(\mathcal{A} \M_{A^\top}\PA_A\g,\mathcal{A}\M_{A^\top}\f)\\
(\mathcal{L} \g,\mathcal{L}\M_{A^\top}\f)\ar@{<=>}[r] &  (\M_{A^\top}\f,\g)\ar@{<=>}[r]\ar@{<=>}[d] & (\A \g,\A\M_{A^\top}\f)\\
(\mathcal{L}\PA_A  \g, \mathcal{L} \f)\ar@{<=>}[r]&( \f,\PA_A \g)\ar@{<=>}[r] & (\A\PA_A \g, \A \f)}$$
where the strong  equivalence notation $(f,g)\Longleftrightarrow(f',g')$  indicates that for the two pairs $(f,g)$ and $(f',g')$, the nonzero eigenvalues and the nonzero variational eigenvalues restricted to $g^{-1}(1)$ % the admissible surface $X=g^{-1}(1)$ 
as  in \eqref{eq:restricted_variationa_eig}  coincide up to some scaling or power factors, and the   corresponding multiplicities (when $p= q$) coincide exactly.  
\end{theorem}

\begin{proof}Theorems \ref{thm:L-transform} and \ref{thm:A-transform} imply that the nonzero spectra of $(\f,\g)$,  $(\mathcal{L}\g,\mathcal{L}\f)$ and $(\mathcal{A}\g,\mathcal{A}\f)$ coincide up to some  scaling or power factors.  
For any $\f\in \CH_p(\R^n)$, let $\tilde{\f}=\f^{\frac1p}$. Then,  $\tilde{\f}\in \CH_1(\R^n)$ and \eqref{eq:ALD} in Proposition  \ref{pro:basic-DAL} implies that 
$$\mathcal{L}\f=l_p(\D \tilde{\f})^{p^*}\text{ and }\mathcal{A}\f=a_p(\D \tilde{\f})^{p}   $$
where $l_p=\frac{p-1}{p^{p^*}}$ and $a_p=\frac{(p-1)^{p-1}}{p^{p}}$ are constants. Then, by Proposition \ref{pro:1-p-q-transfer}, %the above claim, 
the eigenvalue problems of $(\f,\g)$,  $(\mathcal{L}\g,\mathcal{L}\f)$ and $(\mathcal{A}\g,\mathcal{A}\f)$ can be equivalently  reduced to that of $(\tilde{\f},\tilde{\g})$ and $(\mathcal{D}\tilde{\g},\mathcal{D}\tilde{\f})$ up to some scaling factors.  It follows from Theorem \ref{thm:main1} that the spectra %eigenvalue problems 
of $(\tilde{\f},\tilde{\g})$ and $(\mathcal{D}\tilde{\g},\mathcal{D}\tilde{\f})$ coincide  exactly, %are strong equivalent, 
and hence the eigenvalue problems of $(\f,\g)$,  $(\mathcal{L}\g,\mathcal{L}\f)$ and $(\mathcal{A}\g,\mathcal{A}\f)$ are strongly equivalent. %too.

Moreover, according to Theorem \ref{thm:main1} and Corollary \ref{cor:main2}, we have the following strong equivalences regarding norm-like duality:
$$\xymatrix{(\M_{A^\top}\tilde{\f},\M_{A^\top}\PA_A\tilde{\g})\ar@{<=>}[d]\ar@{<=>}[rr]^{\text{norm-like dual\;\;}}&&(\D\M_{A^\top}\PA_A\tilde{\g},\D\M_{A^\top}\tilde{\f})\\ (\M_{A^\top}\tilde{\f},\tilde{\g})\ar@{<=>}[rr]^{\text{norm-like dual\;\;}}\ar@{<=>}[d] & & (\D \tilde{\g},\D\M_{A^\top}\tilde{\f})\\
( \tilde{\f},\PA_A \tilde{\g})\ar@{<=>}[rr]^{\text{norm-like dual\;\;}} & & (\D\PA_A \tilde{\g}, \D \tilde{\f})}$$
Therefore, by Propositions \ref{pro:basic-DAL} and  \ref{pro:1-p-q-transfer}, the strong equivalences among the nine eigenvalue problems shown in the diagram in the statement are established.\qed
\end{proof}

\begin{remark}According to Proposition  \ref{thm:double-dual}, we can  replace 
$\PA_A$ by $\D\M_A\D$ in Theorem \ref{thm:AL}, if we have the additional assumption that  $\mathrm{Ker}\,A\supset \f^{-1}(0)$. 
\end{remark}

\section{Spectral duality for  standard duality transforms}\label{sec:comparison}
While in many applications (see also next Section \ref{sec:app}) it is useful to consider eigenvalue problems with function pairs that have a linear kernel and whose dual is not infinity, in the field of convex analysis or convex geometry it is  frequent to use  the standard version of the  definitions of duality and infimal postcomposition $\hat{\PA}_A\f(\vec x):=A\triangleright f(\vec x)$.  
Note that if we use the latter in place of  $\PA_A$, we can for example   remove the condition $\f^{-1}(0)\subset\mathrm{Ker}(A)$ in Proposition~\ref{thm:double-dual}, that is,  
for any $\f\in \CH_1(\R^n)$, $\forall\vec x\in\R^n$, we have   $\hat{\PA}_A \f(\vec x)=\hat{\D}\M_A\hat{\D}\f(\vec x)$, where $\hat{\D}$ denotes the standard norm dual operator $\hat Df(x) = f^*(x) = \sup_{f(y)\leq 1}\langle x,y\rangle$. 

Let 
$\CVH_p(\R^n)$ be the collection of all   %proper, 
convex, positively $p$-homogeneous functions from $\R^n$ to $[0,+\infty]$.  Clearly, $\CH_p(\R^n)\subsetneqq\CVH_p(\R^n)$. It is known that $\hat{\D}:\CVH_1(\R^n)\to \CVH_1(\R^n)$ and  $\hat{\A}:\CVH_p(\R^n)\to \CVH_p(\R^n)$ are bijections, whereas  $\hat{\mathcal{L}}:\CVH_p(\R^n)\to \CVH_{p^*}(\R^n)$ is a bijection when $p>1$.  A straightforward modification of the proofs of Theorems \ref{thm:main1}, \ref{thm:L-transform} and \ref{thm:A-transform}, leads to the following results.
\begin{theorem}\label{thm:dual-main-o}
For any nonconstant $\f,\g\in \CVH_1(\R^n)$,  the nonzero eigenvalues  of  $(\f,\g)$ and  $(\hat{\D} \g, \hat{\D} \f)$ coincide. Precisely, for any eigenpair $(\lambda,\vec x)$ of $(\f,\g)$ with $\lambda\ne0$ and $\g(\vec x)\ne0$,    $\forall\vec u\in \mathrm{cone}(\partial \f(\vec x))\cap  \mathrm{cone}(\partial \g(\vec x)) $, $(\lambda,\vec u)$ is an eigenpair  of $(\hat{\D} \g,\hat{\D} \f)$. Moreover, if $f$ and $g$ are even functions, then the variational eigenvalues \eqref{eq:variational_eigs} of $(\f,\g)$ and $(\hat{\D} \g,\hat{\D} \f)$ as well as their multiplicities coincide exactly.
\end{theorem}

\begin{theorem}\label{thm:L-transform-o}
Given $p,q>1$, for any functions $f\in\CVH_p(\R^n)$ and $g\in\CVH_q(\R^n)$, for any eigenpair $(\lambda,\vec x)$ of $(\f,\g)$ with $\lambda\ne0$ and $\g(\vec x)\ne0$,    and for any $\vec u\in \mathrm{cone}(\partial \f(\vec x))\cap   \partial \g(\vec x) $, $(\lambda^{p^*-1},\vec u)$ is an eigenpair  of $(\hat{\mathcal{L}} \g,\hat{\mathcal{L}}\f)$. 
\end{theorem}

\begin{theorem}\label{thm:A-transform-o}
For any functions $f\in\CVH_p(\R^n)$ and $g\in\CVH_q(\R^n)$, for any eigenpair $(\lambda,\vec x)$ of $(\f,\g)$ with $\lambda\ne0$ and $\g(\vec x)\ne0$,    and for any $\vec u\in \mathrm{cone}(\partial \f(\vec x))\cap   \partial \g(\vec x) $, $((\frac pq)^{p-2}\frac{(q-1)^{q-1}}{ (p-1)^{p-1}}\lambda,\vec u)$ is an eigenpair  of $(\hat{\A} \g,\hat{\A} \f)$.
\end{theorem}

% Note that, in contrast to Theorems \ref{thm:main1}, \ref{thm:L-transform} and \ref{thm:A-transform}, by using the standard duality operations $\hat{\mathcal{D}}$,  $\hat{\mathcal{L}}$ and $\hat{\mathcal{A}}$ instead of our slightly modified versions, we do not require that $\Ker f$ and $\Ker g$ are linear subspaces. 

\section{Example applications}\label{sec:app}

We devote this final section to discussing a number of problems where discrete nonlinear eigenvalue problems and the nonlinear spectral duality properties developed in the previous sections can be used in application settings from graph and hypergraph optimization  and convex geometry.

\subsection{Nonlinear Laplacians on graphs} 
Let $A\in \R^{n\times m}$ and consider the functions pair $( \|Ax\|_a, \|x\|_b)$, where $\|\cdot \|_a$ and $ \|\cdot\|_b$ are vector norms. Nonlinear eigenvalue problems for this type of convex one-homogeneous functions are among the best studied  problems in nonlinear spectral theory and arise in a broad range of application settings, including inverse problems in imaging  \cite{BGMEC16,Elmoataz,gilboa2018nonlinear}, graph clustering, unsupervised and supervised learning \cite{Bhuler,NIPS2011_193002e6,khrulkov2018art,Tudisco1}, community and core-periphery detection in networks \cite{cristofari2020total,tudisco2018core,tudisco2018community}, graph and hypergraph matching \cite{nguyen2017efficient}.
Note that this type of eigenvalue problems are directly connected with generalized operator matrix norms, which coincide with  the largest (variational) eigenvalue $\lambda_m(f,g) = \max_{x\neq 0}\|Ax\|_a/\|x\|_b=:\|A\|_{a,b}$.

Note that, if $f=\|\cdot \|_a$ and $g = \|\cdot\|_b$, then this type of eigenvalue problem coincides with the eigenvalue problem for the functions pair $(\mathcal M_{A^\top}f,g)$. Based on this observation, in this subsection  we review several  example eigenvalue problems with a direct application to combinatorial optimization problems on finite graphs and discuss what are the various corresponding dual forms. In particular, we will show that several  famous nonlinear graph eigenvalue equations can be recast in various different forms, which has the potential to unleash a variety of new results both from the theoretical and the computational points of view. In fact, established algorithms for the solution of these eigenvalue problems, such as the inverse iteration \cite{HyndLindgren17,jarlebring2014inverse}, the family of ratioDCA methods \cite{NIPS2011_193002e6,tudisco2018community}, or the continuous flow approaches \cite{BungertBurger22,FAGP19}, can be directly transferred to their dual versions and may exhibit improved convergence properties. Moreover, new relations between the graph and the nonlinear eigenpair may be shown. Some of the graph theoretic results presented next are known and properly referenced, others are new and are accompanied by proofs and additional details.

%%% DISCUSS HERE GENERAL RAYLEIGH QUOTIENTS J(x)/\|x\|
%  The nonlinear eigenvalue problem of the form $\lambda \vec x\in \partial J(\vec x)$ was essentially studied in \cite{ABCM01,BCN02}, and the related nonlinear spectral decomposition has been systematically investigated in \cite{BBCN21,BHPG21,BGMEC16}. We would like to focus on the abstract Rayleigh quotient $J(\vec x)/\|\vec x\|^p$ and its related eigenvalue problem $\vec0\in\partial J(\vec x) -\lambda \Phi^p(\vec x)$ \cite{HyndLindgren17,BungertBurger22}, where $J$ is a $p$-homogeneous  convex   functional, and  $\Phi^p:=\frac1p\partial \|\cdot\|^p$. %is an odd set-valued map. 
% By our duality theorems (which can be extended to the infinite dimension case without so much difficulty), this is equivalent to the eigenvalue problem $\vec0\in \lambda \partial J_*(\vec x)-\Phi^q(\vec x)$, where $J_*$ is the   dual of $J$, and $1/p+ 1/q=1$. Then, both the inverse iteration  $\vec0\in \partial J(\vec x^{k+1})- \Phi^p(\vec x^k)$ and the continuous flow $\vec0\in \partial J(\vec x(t)) + \Phi^p(\vec x’(t))$ \cite{HyndLindgren17} and their various analogs \cite{FAGP19,BungertBurger22} can be directly transferred to their dual versions  with the same convergence properties. The resulting dual iteration or dual flow can be used to solve the minimization of the primal abstract Rayleigh quotient.
%%%%%%%%%%%%%%

Before proceeding, we briefly recall some useful graph notation and terminology. A  finite undirected graph $G=(V,E,w)$ is the pair of vertex (or node) set $V=\{1,\dots,n\}$ and edge set $E\subseteq V\times V$, which we equip with a positive weight function $E\ni ij\mapsto w_{ij}>0$. Any such graph is uniquely represented by the incidence matrix $K = (\kappa_{e,u})\in \R^{E\times V}$, which maps any $x\in \R^V$ into the vector with entries $(Kx)_e = \sum_{u}\kappa_{e,u}x_u =  \pm w_{ij}(x_i-x_j)$, where $e=ij\in E$ is the edge connecting nodes $i$ and $j$. Note that the choice of the sign in $(Kx)_e$ is arbitrary but fixed. Different norms of $Kx$ correspond to different energies on $G$. For example, $\|Kx\|_1 = \sum_{ij\in E}w_{ij}|x_i-x_j|$ is the graph total variation, $\|Kx\|_2^2 = \sum_{ij\in E}w_{ij}(x_i-x_j)^2$ the electric potential, $\|Kx\|_\infty = \max_{ij\in E}w_{ij}|x_i-x_j|$ the graph node-wise variation.

\subsubsection{$(1,1)$-Laplacian: Cheeger constant}\label{exam:1-lap}

Let $G=(V,E,w)$ be a weighted  graph  and consider the nonlinear eigenvalue problem 
\begin{equation}\label{eq:1-Lap}
\vec0\in \partial\sum_{ij\in E}w_{ij}|x_i-x_j|-\lambda\partial \sum_{i\in V} |x_i|. 
\end{equation}
 Note that, if $A = K$ is the incidence matrix of $G$, then \eqref{eq:1-Lap} coincides with the eigenvalue problem for the functions pair $(\mathcal M_{A^\top}f,g)$, where  $f  = \|\cdot\|_1$ and $g =\| \cdot  \|_1$ are the standard $l^1$-norms on $\R^E$ and $\R^V$, respectively. Thus, by Corollary \ref{cor:main2} it follows that \eqref{eq:1-Lap} is equivalent to  the following alternative eigenvalue problems
\begin{align}
    & \label{eq:a1}\vec0\in\partial\|  x\|_\infty - \lambda\,  \partial  \inf\limits_{\sum_{ij\in E:j<i} w_{ij}y_{ij}-\sum_{ij\in E:j>i} w_{ij}y_{ij}=x_i}\|\vec y\|_\infty \\
    & \vec0\in\partial \max_{i\in V} \Big|\sum_{ij\in E:j<i} w_{ij}y_{ij}-\sum_{ij\in E:j>i} w_{ij}y_{ij}\Big| -\lambda\,  \partial\|\vec y\|_\infty\\
    & \vec0\in\partial \|\vec y\|_1-\lambda \, \partial\inf\limits_{x:Kx=y}\|x\|_1\\ %\sum_{i\in V} |x_i|\\
    & \label{eq:a4} \vec0\in \partial\sum_{ij\in E}w_{ij}|x_i-x_j|-\lambda\partial \inf\limits_{t\in\R}\sum_{i\in V}|x_i-t|
\end{align}
which correspond to the eigenvalue problems for the pairs $(\D \g,\D\M_{A^\top}\f)$, $(\M_A\D  \g, \D \f)$, $( \f,\PA_A \g)$ and $(\M_{A^\top} \f,\M_{A^\top}\PA_A \g)$, respectively.  
All the above nonlinear eigenvalue problems have the same nonzero eigenvalues (with the same corresponding multiplicities). 

The eigenvalue problem \eqref{eq:1-Lap} is known as the $1$-Laplacian eigenvalue problem on $G$. This is one of the key objects of nonlinear spectral graph theory   and many useful properties of the $1$-Laplacian are known. For example,  when the graph is connected, the smallest  positive eigenvalue of  \eqref{eq:1-Lap} coincides with the Cheeger isoperimetric constant of $G$ \cite{Bhuler,chung1997spectral}. Moreover, when the graph is a tree, each variational eigenvalue of \eqref{eq:1-Lap} coincides with the $k$-th Cheeger constant \cite{deidda2022nodal,Tudisco1}. Precisely, let % it holds 
\begin{equation}\label{eq:Cheeger}
    h_k(G) := \min_{\text{disjoint subsets }V_1,\dots,V_k\subset V}\;\; \max_{1\leq i\leq k} \;\; \frac{\mathrm{vol}(\mathrm{cut}(V_i))}{\mathrm{vol}(V_i)}\, ,
\end{equation}
where $\mathrm{vol}(V_i) = |V_i|$  and  $\mathrm{vol}(\mathrm{cut}(V_i)) = %\frac 1 2
\sum_{u\in V_i, v\notin V_i}w_{uv}$   are the (weighted) volumes of $V_i$ and its cut set, respectively. Then, 
if $\lambda_k$ is the $k$-th variational eigenvalue of the 1-Laplacian \eqref{eq:1-Lap}, it holds $\lambda_2=h_2(G)$ and $\lambda_k=h_k(G)$  for $k>2$ if $G$ is a tree. More in general, we have $\lambda_m \leq h_k(G)\leq \lambda_k$ for a generic  graph $G$, where $m$ is the largest number of nodal domains of any eigenvector of $\lambda_k$ \cite{Tudisco1}. By Theorem \ref{thm:main1} and Corollary \ref{cor:main2}, the same fundamental graph theoretic properties hold for the variational eigenvalues  of each of the nonlinear eigenvalue problems \eqref{eq:a1}--\eqref{eq:a4}. %listed above.

\subsubsection{$(\infty,\infty)$-Laplacian: graph's diameter} 
Let $G=(V,E,w)$ be a weighted graph and consider the so-called $\infty$-Laplacian eigenvalue problem:

\begin{equation}\label{eq:infinite-Lap}
\vec0\in \partial\max_{ij\in E}w_{ij}|x_i-x_j|-\lambda\partial \max_{i\in V}|x_i|.    
\end{equation}
Let $A=K= (\kappa_{e,i})$ be the incidence matrix of the graph, and let $\f:=\|\cdot\|_\infty$ and $\g:=\|\cdot\|_\infty$ be the standard unweighted $l^\infty$-norms on $\R^E$ and  $\R^V$, respectively. Then, \eqref{eq:infinite-Lap} coincides with the nonlinear eigenvalue problem   $\vec0\in \partial_x\|A\vec x\|_\infty-\lambda\partial\|\vec x\|_\infty$, i.e.,   the eigenvalue problem for the functions pair  $(\M_{A^\top}\f,\g)$. By Corollary  \ref{cor:main2}, we obtain several new  eigenvalue problems equivalent to the graph $\infty$-Laplacian: 
\begin{align}
    & \label{eq:l1}\vec0\in\partial\|\vec x\|_1-\lambda \partial \inf\limits_{y: K^\top y = x }\|\vec y\|_1 \qquad \\
    & \label{eq:edge-1-lap}
\vec0\in\partial \sum_{i\in V}\Big|\sum_{e\in E}\kappa_{e,i}y_e\Big|-\lambda \partial\|\vec y\|_1\\
    &  \vec0\in\partial \|\vec y\|_\infty-\lambda \partial\inf\limits_{x:Kx=y}\|\vec x\|_\infty  \qquad \\
    & \label{eq:l4} \vec0\in \partial\max_{ij\in E}w_{ij}|x_i-x_j|-\lambda\partial \Big\|\vec x-\frac{\max_ix_i+\min_ix_i}{2}\mathbf 1\Big\|_\infty 
\end{align}
where $\mathbf 1$ denotes the vector of all ones. 
We emphasize  that the formulation in \eqref{eq:edge-1-lap} corresponds to a form of   1-Laplacian eigenvalue problem on the dual graph, i.e., the eigenvalue problem for the functions pair $f(x) = \|K^\top y\|_1$ and $g(x)= \|y\|_1$.

When the graph is connected, the variational eigenvalues of \eqref{eq:infinite-Lap} are related to the graph diameter. More precisely, define a ball $B=B_r(v)\subseteq V$ centered in $v\in V$ and of radius $r = \mathrm{radius}(B)$ as the set $B_r(v) = \{u \in V:\mathrm{dist}(u,v)\leq r\}$ where  $\mathrm{dist}$ is the shortest path distance on  $G$. Two such balls $B=B_r(v)$ and $B'=B_{r'}(v')$ are  disjoint if $\mathrm{dist}(v,v')\ge r+r'$. With this notation, it holds \cite{infty_lap_preprint} 
\[
\lambda_k\le \min\limits_{\text{disjoint balls }B_1,\dots,B_k\subset V}\;\; \max_{1\le i\le k}\;\;\frac{1}{\mathrm{radius}(B_i)}
\]
where $\lambda_k$ is the $k$-th variational eigenvalue of \eqref{eq:infinite-Lap}. %or any of the eigenvalue problems \eqref{eq:l1}--\eqref{eq:l4}. 
In particular, note that the  smallest nonzero variational eigenvalue coincides with $2/\mathrm{diam}(G)$, where $\mathrm{diam}(G):=\max_{i,j\in V}\mathrm{dist}(i,j)$, and $\mathrm{dist}$ represents the shortest path distance on $G$. More precisely, if $G$ has $k$ connected components, $G_1,\cdots,G_k$,  then the smallest positive variational eigenvalue coincides with 
\[
\min\limits_{i=1,\cdots,k}\frac{2}{ \mathrm{diam}(G_i)}=\frac{2}{\max\limits_{i=1,\cdots,k}\mathrm{diam}(G_i)}.
\]
By Corollary \ref{cor:main2}, all the above properties transfer directly to the nonlinear spectrum of any of the eigenvalue problems \eqref{eq:l1}--\eqref{eq:l4}.

\begin{remark}
The cycle graph $C_n$ is the only graph which is dual to itself, i.e., is such that $K=K^\top$. If we work on a cycle graph, the 1-Laplacian eigenvalue problem  \eqref{eq:1-Lap} is equivalent to the $\infty$-Laplacian eigenvalue problem  \eqref{eq:infinite-Lap}, via the spectral duality equivalence shown in \eqref{eq:edge-1-lap}. In particular, their $k$-th variational eigenvalues coincide, and they are bounded by the $k$-th  Cheeger constant $h_k(G)$ which is consistent with  the reciprocal  of the largest radius of any ball in any set of  $k$ pairwise disjoint balls inside the cycle graph.

It is interesting to note that in a Euclidean space, a ball $B$ of radius $r$ satisfies  $\frac{\mathrm{vol}(\partial B)}{\mathrm{vol}(B)}\sim \frac1r$, where $\partial B$ is the boundary of $B$ and  $\sim$ denotes here that the two quantities are proportional. As $\mathrm{cut}$ is the graph analogue of the boundary, an interesting open question is whether or not $h_k(G)\sim \frac 1 {r_k(G)}$ for a generic graph $G$, where $r_k(G)$ denotes the largest radius of any ball in any set of  $k$ pairwise disjoint balls in the graph.
\end{remark}

%\begin{example}[$1/\infty$-Laplacian: maxcut and mincut] 
\subsubsection{$(1,\infty)$-Laplacian: maxcut and mincut}

Let $G=(V,E,w)$ be a weighted graph. Consider the eigenvalue problem for % either of 
the functions pair $(\mathcal M_{A^\top}f,g)$ with  $f(x) = \|x\|_1$, $g(x) = \|x\|_\infty$ and $A=K$, namely 
\begin{equation}\label{eq:1/infinite}
\vec0\in \partial\sum_{ij\in E}w_{ij}|x_i-x_j|-\lambda\partial\|\vec x\|_\infty\, .
\end{equation}
By our spectral duality principle in Theorem \ref{thm:main1} and Corollary \ref{cor:main2},  \eqref{eq:1/infinite} is equivalent to 
\begin{align}
    & \label{eq:ll1}\vec0\in\partial\|\vec x\|_1-\lambda \partial \inf\limits_{\sum_{j<i} w_{ij} y_{ij}-\sum_{j>i} w_{ij}y_{ij}=x_i}\|\vec y\|_\infty \\
    & \vec0\in\partial \sum_{i\in V}\Big|\sum_{j<i} w_{ij}y_{ij}-\sum_{j>i} w_{ij}y_{ij}\Big|-\lambda \partial\|\vec y\|_\infty\\
    &  \vec0\in\partial \|\vec y\|_1-\lambda \partial\inf\limits_{x:Kx=y}\|\vec x\|_\infty \\
    & \label{eq:ll14} \vec0\in \partial\sum_{ij\in E} w_{ij}|x_i-x_j|-\lambda\partial \Big\|\vec x-\frac{\max_ix_i+\min_ix_i}{2}\mathbf 1\Big\|_\infty \, .
\end{align}

It is shown in \cite[Section 4.2]{JostZhang21-} that the smallest nonzero variational eigenvalue and the largest variational eigenvalue of \eqref{eq:1/infinite} coincide with  the mincut  and the maxcut values of $G$, respectively defined as 
\[
\mathrm{mincut}(G) = \min_{S\subset V}\mathrm{vol}(\mathrm{cut}(S))\quad \text{and} \quad \mathrm{maxcut}(G) = \max_{S\subset V}\mathrm{vol}(\mathrm{cut}(S))\, .
\]

We also remark that (a)   $\mathrm{maxcut}(G)$ is actually equivalent to the largest eigenvalue for the pair $(\|K\vec x\|_p,\|\vec x\|_\infty)$, for any $1\le p<\infty$,  see  Example 3.1 and Section 4.2 in \cite{JostZhang21-}; and (b) when $w_{ij}$ in \eqref{eq:1/infinite} is replaced by the modularity weights $m_{ij} := d_id_j/(\sum_k d_k) - w_{ij}$, $d_i=\sum_jw_{ij}$, the largest eigenvalue of \eqref{eq:1/infinite} corresponds to the leading community in $G$, see Theorem 3.7 in \cite{tudisco2018community}. 
Due to the nonlinear spectral duality principle, the same properties hold for each of the nonlinear eigenvalue problems in \eqref{eq:ll1}--\eqref{eq:ll14}.
Moreover, the following relation holds for their $k$-th variational eigenvalue
\begin{theorem}
    Let $\lambda_k$ be the $k$-th variatonal eigenvalue of the eigenvalue problem  \eqref{eq:1/infinite}. %any of the eigenvalue problems \eqref{eq:1/infinite}--\eqref{eq:ll14}. 
    Then 
    $$
\lambda_k \le\; \min\limits_{V_1,\cdots,V_k\text{ form a partition of }V}\;\;\mathrm{maxcut}(G[V_1,\cdots,V_k]) 
$$
where $\mathrm{maxcut}(G[V_1,\cdots,V_k]):=2\max\limits_{S\subset\{1,\cdots,k\}}\sum\limits_{i\in S,j\in V\setminus S}w_{V_i,V_j}$ denotes the maxcut value of the  graph $G[V_1,\cdots,V_k]$, formed by $k$ vertices  corresponding to the $k$ sets $V_1,\cdots,V_k$,  with edge weights 
\[
w_{V_i, V_j}=\sum_{a\in V_i,b\in V_j}w_{ab}, \qquad V_i\neq V_j\, .
\]
\end{theorem}
\begin{proof}
For any partition $(V_1,\cdots,V_k)$ of $V$, denote by  $1_{V_i}$ the indicator vector of $V_i$. Then $1_{V_1}$, $\cdots$, $1_{V_k}$ are linearly independent. Thus $\mathrm{genus}(\mathrm{span}(1_{V_1},\cdots,1_{V_k}))= k$ and we have 
\begin{align*}\lambda_k(\mathcal M_{A^\top}f,g)&\le \max\limits_{\vec x\in\mathrm{span}(1_{V_1},\cdots,1_{V_k})}\frac{\sum\limits_{\{i,j\}\in E}w_{ij}|x_i-x_j|}{\|\vec x\|_\infty}
\\&= \max\limits_{(t_1,\cdots,t_k)\in \R^k\setminus\{0\}}\frac{\sum\limits_{1\le i<j\le k}w_{V_i,V_j}|t_i-t_j|}{\max\limits_{i=1,\cdots,k}|t_i|}
\\&= 2\max\limits_{S\subset\{1,\cdots,k\}}\sum\limits_{i\in S,j\in V\setminus S}w_{V_i,V_j},
\end{align*}
 where the last equality follows from  Theorem 4.1 in \cite{JostZhang21-}.
\qed\end{proof}

\subsubsection{$(\infty,1)$-Laplacian: graph's inscribed ball}
Let $G=(V,E,w)$ be a weighted  graph. Consider the eigenvalue problem for the  functions pair $f(x) = \|Kx\|_\infty$ and $g(x)=\|x\|_1$, namely

\begin{equation}
\label{eq:infinite/1}\vec0\in \partial\max_{ij\in E} w_{ij}|x_i-x_j|-\lambda\partial\|\vec x\|_1\, .
\end{equation}
Then, by Theorem \ref{thm:main1} and Corollary \ref{cor:main2}, \eqref{eq:infinite/1} is equivalent to 
\begin{align}
    & \vec0\in\partial\|\vec x\|_\infty-\lambda \partial \inf\limits_{\sum_{j<i}  w_{ij}y_{ij}-\sum_{j>i}  w_{ij}y_{ij} = x_i,\forall i}\|\vec y\|_1   \\
    & 
\vec0\in\partial \max_{i\in V}|\sum_{j<i}  w_{ij}y_{ij}-\sum_{j>i}  w_{ij}y_{ij}|-\lambda \partial\|\vec y\|_1\\
    &  \vec0\in\partial \|\vec y\|_\infty-\lambda \partial\inf\limits_{x:Kx=y}\|\vec x\|_1\\
    & \label{eq:g4} \vec0\in \partial\max_{ij\in E}w_{ij}|x_i-x_j|-\lambda\partial \Big\|\vec x-\frac{\max_ix_i+\min_ix_i}{2}\mathbf 1\Big\|_1 
\end{align}
% \begin{align}
%     &\vec0\in\partial \max_{i\in V}|\sum_{j<i}  w_{ij}y_{ij}-\sum_{j>i}  w_{ij}y_{ij}|-\lambda \partial\|\vec y\|_1.
% \end{align}
Moreover, the following result holds for the variational eigenvalue of all the above eigenvalue problems.
\begin{theorem}
    Let $\lambda_k$ be the $k$-th variational eigenvalue of the eigenvalue equation \eqref{eq:infinite/1}. %any of the eigenvalue equations \eqref{eq:infinite/1}--\eqref{eq:g4}.
    It holds
\begin{equation}\label{eq:infty/1-k}
\lambda_k \le \min\limits_{\text{disjoint balls }B_1,\cdots,B_k\subset V}\;\;\max_{1\le i\le k}\;\frac{1}{\mathrm{size}(B_i)}    
\end{equation}
where, for $B = B_r(v)$ we let 
$\mathrm{size}(B) = \sum_{i=0}^r (r-i) |\{u \in V:\mathrm{dist}(u,v)= i\}|$. 
\end{theorem}
\begin{proof}
  For a  $x\in \R^n$ and a  ball $B$ with radius $r$ and  centered at the vertex $v$ define $\vec x^B\in\R^n$ by $(\vec x^B)_i=\max\{r-\mathrm{dist}(v,i),0\}$. Then, for any $k$ disjoint balls $B_1,\cdots,B_k\subset V$, $\vec x^{B_1},\cdots,\vec x^{B_k}$ are linearly independent. Thus $\mathrm{genus}(\mathrm{span}(\vec x^{B_1},\cdots,\vec x^{B_k}))\geq k$ and we have 
\begin{align*}
    \lambda_k&\le \max\limits_{\vec x\in\mathrm{span}(\vec x^{B_1},\cdots,\vec x^{B_k})}\frac{\max\limits_{\{i,j\}\in E}|x_i-x_j|}{\|\vec x\|_1}\\
    &\le \max_{1\le s\le k}\frac{\max\limits_{\{i,j\}\in E}|x_i^{B_s}-x_j^{B_s}|}{\|\vec x^{B_s}\|_1}=\max_{1\le s\le k}\frac{1}{\mathrm{size}(B_s)}.
\end{align*}
  where the second inequality follows from the fact that the $x^{B_j}$ have disjoint support. By taking the minimum over all possible disjoint balls we obtain \eqref{eq:infty/1-k}.
\qed\end{proof}
Note that, as a consequence of the above theorem we obtain that the smallest eigenvalue  is at most the reciprocal of the size of the largest ball inscribed in the graph.

\subsubsection{Hypergraphs and core-periphery detection}

On top of combinatorial problems on graphs,  nonlinear eigenvalue problems appear in a variety of  hypergraph mining settings where the optimization of suitable discrete functions is required. Nonlinearity is particularly important when we deal with a hypergraph, as the presence of higher-order node interactions naturally leads to nonlinear eigenvalue equations and the corresponding nonlinear operators. Examples include submodular and diffusion-inspired hypergraph Laplacians \cite{chan2018spectral,li2018submodular}, tensor-based Laplacians \cite{gautier2019unifying,hu2015laplacian}, and game-theoretic homogeneous Laplacians \cite{flores2022analysis}. To provide a concrete example, we consider here the  core-periphery detection problem on hypergraphs, as formulated in  \cite{TudiscoHigham22}. 

Consider a hypergraph $H=(V,E,w)$ made by a set of vertices $V=\{1,\ldots,n\}$, hyperedges  $E=\{e_{1},\cdots,e_{m}\}$ and the weight function $w:E\to \R_+$. Here, unlike the graph case, each $e\in E$ contains an arbitrary number of nodes. 
The core-periphery detection problem consists of identifying the optimal subdivision of $V$ into a core set highly connected with the rest of $H$ and a periphery set, connected only (or mostly) to the core.

It is shown in \cite{TudiscoHigham22} that this combinatorial problem on $H$ boils down to the norm-constrained
optimization problem, 
\begin{equation}\label{eq:cp}
    \max \sum_{e\in E}w_{e}\|x|_e\|_q\quad \text{s.t.}\quad \|x\|_p=1\, .
\end{equation}
Clearly, if $f(\vec x)=\sum_{e\in E}w_{e}\|\vec x|_e\|_q$ and $g(\vec x)=\|\vec x\|_p$, the above problem coincides with the   largest eigenvalue of the nonlinear eigenvalue problem for the functions pair $(f,g)$.  Now, we shall write down the dual eigenvalue problem, i.e., the eigenvalue problem for the   function pair $(\D g, \D f)$.

For $g$ we have $\D g(\vec x)=\|\vec x\|_{p^*}$, where $1/p+1/p^*=1$. As for $f$, note that 
\[f(\vec x)=\|(\|\vec x|_{e_1}\|_q,\ldots,\|\vec x|_{e_m}\|_q)\|_{1,w}\, ,
\]
where $\|\cdot\|_{1,w}$ indicates the weighted $l^1$-norm on $\R^E$. %, namely,  $\|\vec y\|_{1,w}:=\sum_{e\in E}w_{e}|y_e|$, for $\vec y\in\R^E$.  
Then, by Proposition \ref{pro:compo-norm}, we have 
\[\D f(\vec x)=\inf\limits_{\substack{\sum_{e\in E}\vec y_{e}=\vec x\\ \mathrm{supp}(\vec y_{e})\subset e}}\left\|\Big( \frac{\|\vec y_{e_1}\|_{q^*}}{w_{e_1}},\ldots,\frac{\|\vec y_{e_m}\|_{q^*}}{w_{e_m}}\Big)\right\|_{\infty}=\inf\limits_{\substack{\sum_{e\in E}\vec y_e=\vec x\\ \mathrm{supp}(\vec y_e)\subset e}}\max\limits_{e\in E}\frac{\|\vec y_e\|_{q^*}}{w_{e}}
\]
where $\vec y_e$ denotes a  vector in $\R^V$ with the support in $e$. 

Moreover, using  Corollary \ref{cor:main2} we can obtain additional equivalent formulations.  
For $e\in E$, let $A_e:\R^V\to \R^V$ be a matrix   defined as  $(A_e\vec x)_i=x_i$ if $i\in e$ and $(A_e\vec x)_i=0$ if $i\not\in e$.  Clearly, $\|A_e\vec x\|_q=\|\vec x|_e\|_q$. Thus, we can write $f$ as $f=\tilde{f}\circ A$, i.e.,  $f(\vec x)=\tilde{f}( A\vec x)$, where $\tilde{f}:\R^{nm}\to[0,+\infty)$ is the norm defined as 
\[\tilde{f}(\vec y_{1},\dots,\vec y_{m})=\|(\|\vec y_{1}\|_q,\ldots,\|\vec y_{m}\|_q)\|_{1,w}
\]
with $\vec y_{j}\in\R^n$ and $A:\R^n\to \R^{nm}$  defined as $A\vec x=(A_{e_1}\vec x,\dots,A_{e_m}\vec x)$. Thus, we immediately see that 
\[\D \tilde{f}(\vec y_{1},\dots,\vec y_{m})=\|(\|\vec y_{1}\|_{q^*}/w_{e_1},\ldots,\|\vec y_{m}\|_{q^*}/w_{e_m})\|_{\infty}
\]
for any vector $(\vec y_{1},\dots,\vec y_{m})$ of dimension $n\times m$.  By Corollary \ref{cor:main2}, the largest eigenvalue of the dual eigenvalue problem  $(\D\g\circ A^\top,\D\tilde{f})$, i.e.,
\[
0\in \partial \|\sum_{i=1}^mA_{e_i}^\top\vec y_i%|_{e_i}
\|_{p^*} -\lambda \partial \left\|\Big( \frac{\|\vec y_{1}\|_{q^*}}{w_{e_1}},\ldots,\frac{\|\vec y_{m}\|_{q^*}}{w_{e_m}}\Big)\right\|_{\infty}
\]
coincides with the core-periphery eigenvalue problem \eqref{eq:cp} for $(f,g)$.

\subsection{Distance between  convex bodies}

The Banach-Mazur distance is a key quantity in  convex geometry and functional analysis, which has led  to noteworthy progress in both those areas, see e.g.\ \cite{JS21}. 
Here, we focus on the  Banach-Mazur distance in its multiplicative form between two centrally symmetric convex bodies $K$ and $L$,  centered at the origin point $\vec 0$ in $\R^n$. This distance is defined as 
\begin{equation}\label{eq:BM-distance}
d(K,L)=\inf\{r\ge 1:\exists A\in GL(\R^n)\text{ s.t. } L\subset AK\subset rL\},    
\end{equation}
where  $GL(\R^n)$  is the  general linear group. 
By translating our spectral duality properties into the language of convex geometry we can immediately obtain properties about this distance between convex bodies via the eigenvalue problem for function pairs. 

For two convex bodies $K$ and $L$ containing the origin as an  interior  point, there exists some scaling constant $\lambda>0$ such that the two convex surfaces  $\partial K$ and  $\lambda \partial L$ are tangent to each other at some  point, where  $\partial K$ and $ \partial L$ are the boundary surfaces of the bodies $K$ and $L$, respectively.  Here, we say that two convex surfaces are tangent at $\vec a$ if they have a  common  supporting hyperplane at $\vec a$. Let 
\[
\mathcal{ST}(K,L)=\{\lambda>0 : \partial K\text{ and }\lambda\partial L\text{ are tangent}\}\, .
\]
A first key observation is  
that $\mathcal{ST}(K,L)$ is a compact subset of $(0,+\infty)$ and it coincides with the set of all the nonzero eigenvalues of the function pair  $(\|\cdot\|_K,\|\cdot\|_L)$, where $\|\cdot\|_K$ is the  Minkowski functional norm of $K$, i.e., the norm such that  $K = \{\vec x\in\R^n:\|\vec x\|_K\le 1\}$. Also, it is easy to see that $\mathcal{ST}(L,K)=\{\lambda^{-1}:\lambda\in \mathcal{ST}(K,L)\}$. For such a pair of convex bodies $K$ and $L$, we can still  use \eqref{eq:BM-distance} to define their simple Banach-Mazur distance and use our spectral duality to introduce new distances and observe new identities.

Preciesely, let $\lambda_{\max}(K,L)=\max\{\lambda: \lambda \in \mathcal{ST}(K,L)\}$ and $\lambda_{\min}(K,L)=\min\{\lambda: \lambda \in \mathcal{ST}(K,L)\}$. Corollary \ref{cor:main2}  implies that $\mathcal{ST}(A^\top L^*,K^*)=\mathcal{ST}(AK,L)$ for any $n\times n$ invertible matrix $A$. Thus, we have the following new representation of the Banach-Mazur distance 
\begin{align*}
d(K,L)&=
\inf_{A\in GL(\R^n)}  \frac{\lambda_{\max}(AK,L)}{\lambda_{\min}(AK,L)}=\min_{A\in SL(\R^n)}  \frac{\lambda_{\max}(AK,L)}{\lambda_{\min}(AK,L)} \\
&=\min_{A\in SL(\R^n)} \lambda_{\max}(AK,L)\lambda_{\max}(A^\top L^*,K^*)
\end{align*}
where  $SL(\R^n)$ indicates  the special linear group, i.e., the set of matrices with determinants equal to one. 

From this formulation, we immediately obtain $d(K,L)=d(K^*,L^*)$, which generalizes the known equality for symmetric convex bodies to the nonsymmetric case. In fact, by Corollary \ref{cor:main2}, $\lambda_{\max}(A^\top L^*,K^*)=\lambda_{\max}(AK,L)$ and $\lambda_{\min}(A^\top L^*,K^*)=\lambda_{\min}(AK,L)$, and thus,
$$d(L^*,K^*)=
\inf_{A^\top\in GL(\R^n)}  \frac{\lambda_{\max}(A^\top L^*,K^*)}{\lambda_{\min}(A^\top L^*,K^*)}=\inf_{A\in GL(\R^n)}  \frac{\lambda_{\max}(AK,L)}{\lambda_{\min}(AK,L)}=d(K,L).
$$

Using a similar argument we can obtain a similar result for other distances. In particular, consider the distance defined by \[
\widehat{d}(K,L)=\inf\Big\{r\ge 1: \frac1rL\subset K\subset rL\Big\}\, .
\]
This distance is used for studying floating and illumination  bodies  \cite{MW19},  and is   equivalent to the  Goldman-Iwahori metric introduced for Bruhat-Tits buildings \cite{Haettel22}. 
We have 
\[
\widehat{d}(K,L)=\max\{\lambda_{\max}(K,L),\frac{1}{\lambda_{\min}(K,L)}\}=\max\{\lambda_{\max}(K,L), \lambda_{\max}(L,K)\}
\]
and from this new representation, we can easily obtain  the duality identity $\widehat{d}(K^*,L^*)=\widehat{d}(K,L)$ via Theorem \ref{thm:main1} %, Corollary \ref{cor:main2}
and the discussion above.

%\begin{acknowledgements}
%\end{acknowledgements}


\end{document}